\documentclass[a4paper,10pt]{article}
\usepackage{ucs} 
\usepackage[utf8x]{inputenc}
\usepackage{amsmath,amsfonts,amssymb,amsthm,amsopn,amscd}
\usepackage{bbm}
\usepackage{graphicx}
\usepackage[colorlinks=true]{hyperref}
\usepackage{ifthen}
\usepackage{fullpage}
\usepackage{epic}
\newtheorem{proposition}{Proposition}
\newtheorem{lemma}{Lemma}
\newtheorem{definition}{Definition}
\newtheorem{theorem}{Theorem}

\newtheorem{corollary}{Corollary}
\usepackage{authblk}
\usepackage{times}
\usepackage{subfigure}

\newcommand{\T}{\mathcal{T}}
\newcommand{\tT}{\tilde\T}
\newcommand{\mH}{\mathcal{H}}  
\newcommand{\R}{\mathbb{R}}
\newcommand{\Z}{\mathbb{Z}}
\newcommand{\C}{\mathbb{C}}
\newcommand{\N}{\mathbb{N}}
\newcommand{\tx}{\tilde{x}}
\newcommand{\mr}{\mathbf{r}}  
\newcommand{\me}{\mathbf{e}}
\newcommand{\mk}{\mathbf{k}}
\newcommand{\arctanh}{{\rm arctanh}}
\newcommand{\D}{\mathcal{D}}
\newcommand{\Oct}{\mathcal{O}}
\newcommand{\DR}{\mathcal{D}\times\R^+_*}
\newcommand{\bz}{{\bf z}}
\newcommand{\GL}{\rm{GL(2,\R)}}

\newboolean{DisplaySOS}
\setboolean{DisplaySOS}{true} 
\newcommand{\SOS}[1]{\ifthenelse{\boolean{DisplaySOS}}{{\textcolor{red}{\bf[#1]}}}{}}

\title{Bifurcation of hyperbolic planforms}
\author[1,2]{Pascal Chossat}
\author[1]{Gr\'egory Faye}
\author[1]{Olivier Faugeras}
\affil[1]{\small NeuroMathComp Laboratory, INRIA, Sophia Antipolis, CNRS, ENS Paris, France}
\affil[2]{\small Dept. of Mathematics, JAD Laboratory, CNRS and University of Nice Sophia-Antipolis, Parc Valrose, 06108 Nice Cedex 02, France}

\begin{document}

\maketitle

\begin{abstract}
Motivated by a model for the perception of textures by the visual cortex in primates, we analyse the bifurcation of periodic patterns for nonlinear equations describing the state of a system defined on the space of structure tensors, when these equations are further invariant with respect to the isometries of this space. We show that the problem reduces to a bifurcation problem in the hyperbolic plane $\D$ (Poincar\'e disc). We make use of the concept of periodic lattice in $\D$ to further reduce the problem to one on a compact Riemann surface $\D/\Gamma$, where $\Gamma$ is a cocompact, torsion-free Fuchsian group. The knowledge of the symmetry group of this surface allows to carry out the machinery of equivariant bifurcation theory. Solutions which generically bifurcate are called "H-planforms", by analogy with the "planforms" introduced for pattern formation in Euclidean space. This concept is applied to the case of an octagonal periodic pattern, where we are able to classify all possible H-planforms satisfying the hypotheses of the Equivariant Branching Lemma. These patterns are however not straightforward to compute, even numerically, and in the last section we describe a method for computation illustrated with a selection of images of octagonal H-planforms.
\end{abstract}


\section{Introduction}\label{section:introduction}

In a recent paper \cite{chossat-faugeras:09} a model for the visual perception of textures by the cortex was proposed, which assumes that populations of neurons in each so-called "hypercolumn" of the visual cortex layer V1 are sensitive to informations carried by the structure tensor of the image (in fact, of a small part of it since each hypercolumn is dedicated to a specific part of the visual field). The structure tensor is a $2\times 2$ positive definite, symmetric matrix, the eigenvalues of which characterize to some extent the image geometric properties like the presence and position of an edge, the contrast, etc. This led to write the average membrane potential $V$ of the populations of neurons in a hypercolumn as a function of the structure tensor (and time). The equations that govern the evolution of the average membrane potential are
\begin{equation} \label{eq:neuralmass}
\frac{\partial V}{\partial \tau}(\T,\tau)=-\alpha V(\T,\tau)+\int_\mH w(\T,\T')S(V(\T',\tau))\,d\T'+I(\T,\tau)
\end{equation}
where $\T,~\T'\in SDP(2)$ (the space of structure tensors), $S$ is a smooth function $\R\rightarrow\R$ of sigmoid type ($S(x)\rightarrow \pm 1$ as $x\rightarrow \pm\infty$), $I$ corresponds to some input signal coming from different brain areas such as the thalamus and $w(\T,\T')$ is a function which expresses the interaction between the populations of neurons
of types $\T$ and $\T'$ in the hypercolumn. We can introduce a distance on the set of structure tensors by noting that it has the structure of a homogeneous space for the group of invertible matrices $\GL$ acting by coordinate change on quadratic forms associated with the structure tensors \cite{weil:57}: for any $\T\in SDP(2)$ and $G \in \GL$, we define $G\cdot \T =~^tG~\T~G$. This induces a Riemannian metric on the tangent space: $ g_\T( A, B )={\rm tr}(\T^{-1} A \T^{-1} B)$ and the corresponding distance in $SDP(2)$ is given by (see \cite{moakher:05})
\begin{equation}  \label{distanced0}
d_0(\T_1,\T_2) = \sqrt{\log^2 \lambda_1+\log^2 \lambda_2}
\end{equation}
where $\lambda_1$, $\lambda_2$ are the eigenvalues of $\T_1^{-1}\T_2$. Because of its invariance this distance is biologically plausible since the neurons have no obvious way of knowing in which coordinate system are expressed the components of the structure tensors and this leads to the assumption
that the function $w$ is insensitive to coordinate changes in $SDP(2)$. This implies that $w$ is in fact a function of the distance $d_0$: 
\begin{equation} \label{Invariance hypothesis} 
w(\T,\T') = f(d_0(\T,\T'))
\end{equation}
As a consequence, the equation (\ref{eq:neuralmass}) is invariant under the action of the isometry group $\GL$. 

Assume now that there is no input: $I=0$, and that $V=0$ is a solution of (\ref{eq:neuralmass}) (we can assume w.l.o.g. that $S(0)=0$). If the slope $\mu=S'(0)$ is small enough, small perturbations of this basic state are damped to 0 and the basic state is stable. We can however expect that if $\mu$ exceeds a critical value an instability will grow and lead to a new state which need not be invariant under the action of $\GL$. In other word we expect {\em pattern formation} through spontaneous symmetry breaking, a phenomenon which has been widely studied in other contexts in the last 40 years (see \cite{hoyle:06} for a review). From the mathematical point of view the theoretical framework is the well-established bifurcation theory with symmetry \cite{golubitsky-stewart-etal:88}, \cite{chossat-lauterbach:00}. There is however a basic difference between the  problem at hand and previous studies in pattern formation and symmetry-breaking bifurcation: the symmetries (isometries) are not Euclidean. Moreover the group $\GL$ being non compact, the spectrum of the equivariant linear operators is continuous and eigenvalues have infinite multiplicity in general. This is also true when the symmetry group is the group of displacements and reflections in Euclidean space. In this case a rational approach is to resctrict the analysis to the bifurcation of patterns which are periodic (in space) with a period equal to $2\pi/k$ where $k$ is the wave number of the most unstable Fourier modes. This hypothesis is supported by observations (to some extent) and it provides a framework which is suitable for the bifurcation analysis. Namely, the spectrum of operators restricted to functions which are invariant on a periodic lattice is discrete and consists of eigenvalues with finite multiplicity. This allows to fully exploit the power of equivariant bifurcation theory and to describe the different types of bifurcated solutions by their (remaining) symmetries. 

We would like to proceed in a similar way for equation (\ref{eq:neuralmass}), that is to consider solutions which are periodic in the space of structure tensors (one has to give a precise meaning to this statement) and then apply equivariant bifurcation methods to describe the set of solutions. We are not addressing here the question of the relevance of looking for such solutions, neither from the neurophysiological point of view, nor even from the point of view of their stability under perturbations which do not have the same periodicity (or no periodicity at all). Let us mention the fact that in \cite{chossat-faugeras:09} we identified families of subgroups of the group of isometries of the set of structure tensors that naturally arose from the analysis of the retinal input to the hypercolumns in visual area V1. These subgroups, which we called neuronal Fuchsian groups, have the property that the membrane potential functions are invariant under their action. This is one example of the neurophysiological relevance of looking for solutions of \eqref{eq:neuralmass} that are periodic in this sense but it certainly does not give the final answer: These difficult questions will be addressed subsequently. Let us point out that, on the other hand, a classification of possible periodic patterns should be largely independent of the model equations as long as these equations share some basic properties (like $\GL$ invariance). In this paper we shall therefore not focus on equation (\ref{eq:neuralmass}) and our results would apply equally to other models such as reaction-diffusion or Swift-Hohenberg equations with Laplace-Beltrami operator in $SDP(2)$). 

The structure of the paper is as follows:
\begin{enumerate}
\item In section \ref{section:periodiclattices} we introduce the necessary basic material from hyperbolic geometry. Periodic lattices and functions are defined in $SDP(2)$ and it will be shown that the problem can be decoupled and reduced to looking for periodic patterns in the hyperbolic plane (we shall work with its Poincar\'e disc representation $\D$). 
\item In section \ref{section:bifurcation} we set the bifurcation problem for periodic patterns in $\D$. Applying the equivariant branching lemma \cite{golubitsky-stewart-etal:88} (or isolated stratum principle in the variational case), this leads us to defining {\em hyperbolic planforms}, which are the hyperbolic counterparts of the planforms defined in the Euclidean case (see \cite{dionne-golubitsky:92}).
\item In section \ref{section:octagon} the methods of section \ref{section:bifurcation} are applied to the case of a regular octagonal pattern for which we can describe all planforms which result from the applicaiton of the equivariant branching lemma. 
\item In section \ref{section:computing} we compute eigenvalues and eigenfunctions of the Laplace-Beltrami operator in the hyperbolic octagon in order to exhibit {\em hyperbolic planforms} which satisfy certain isotropy conditions.
\end{enumerate}

\section{Periodic lattices and functions in the space of structure tensors}\label{section:periodiclattices}

\subsection{The space SDP(2)}

The following decomposition will prove to be very convenient. Let $\xi^2$ be the determinant of $\T\in SDP(2)$. We set $\xi>0$ by convention.  Writing $\T=\xi\T'$, we see that $SDP(2)=\R^+_*\times SSDP(2)$ where $SSDP(2)$ denotes the subspace of tensors with determinant 1. In the open cone $SDP(2)$ the surface $SSDP(2)$ is a hyperboloid sheet and it can be shown that it carries a metric induced by the metric $g$ which is just the usual metric of the hyperbolic plane. The isometry group of displacements in $SSDP(2)$ is the special linear group $SL(2, \R)$, and indeed we may also write the group of orientation preserving isometries $\rm{GL^+(2,\R)}=SL(2,\R)\times \R^+_*$ acting on a tensor $\T=(\T',\xi)$ by
\begin{equation} \label{action}
(\Gamma,\alpha)\cdot (\T',\xi) = ( ~^t\Gamma \T'\Gamma,\alpha^2\xi).
\end{equation}

We now identify $SDP(2)$  with the "half" open cylinder $\DR$, where $\D$ is the Poincar\'e disc, through the following change of variables. \\ Let $\T=\left( \begin{array}{cc}
	 x_1 & x_3 \\
	 x_3  & x_2
	\end{array} \right)$ with $x_1>0$ and $x_1x_2-x_3^2>0$. We set 
$$x_i=z_3\tilde x_i,~i=1,2,3,~z_3>0,~\tx_1\tx_2-\tx_3^2=1.$$
Now the hyperbolic plane $SSDP(2)$ is further identified with $\D$ through the change of coordinates
\begin{equation}
 \left\{
\begin{array}{lcl}
\tx_1 & = & \frac{(1+z_1)^2+z_2^2}{1-z_1^2-z_2^2}\\
\tx_2 & = & \frac{(1-z_1)^2+z_2^2}{1-z_1^2-z_2^2}\\
\tx_3 & = & \frac{2z_2}{1-z_1^2-z_2^2}
\end{array}
\right.
\end{equation}
These formulas define a diffeomorphism
\begin{equation} \label{diffeoTheta}
\Theta:~(x_1,x_2,x_3)\in SDP(2) \mapsto (z_1,z_2,z_3)\in \D\times\R^+_*
\end{equation}
The group of direct isometries (displacements) in $\D\times\R^+_*$ is $SU(1,1)\times\R^+_*$ where $SU(1,1)=\left( \begin{array}{cc}
	 \alpha & \beta \\
	 \bar\beta  & \bar\alpha
	\end{array} \right)$ with $|\alpha|^2-|\beta|^2=1$. Including orientation reversing isometries would replace $SU(1,1)$ by $U(1,1)$. 

\begin{proposition}
The metric carried to $\D\times\R^+_*$ by $\Theta$ is (up to a normalization coefficient)
\begin{equation}
g_1(z_1,z_2,z_3) = \frac{4(dz_1^2 + dz_2^2)}{(1-z_1^2-z_2^2)^2} + \frac{dz_3^2}{z_3^2}
\end{equation}
\end{proposition}
\begin{proof}
The fact that the decomposition is diagonal follows directly from the direct product decomposition and group action (\ref{action}). In order to compute the precise expression for the metric, let us first express the form $g$ in $SDP(2)$ in suitable coordinates of the tangent space (which is the space of symmetric $2\times 2$ matrices). Such a basis is given, at any tensor $\T$, by
\[
\frac{\partial}{\partial x_1}=
	\left[
	\begin{array}{cc}
	 1 & 0\\
	 0  & 0
	\end{array}
	\right] \quad 
\frac{\partial}{\partial x_2}=
	\left[
	\begin{array}{cc}
	 0 & 0\\
	 0  & 1
	\end{array}
	\right] \quad
\frac{\partial}{\partial x_3}=
	\left[
	\begin{array}{cc}
	 0 & 1\\
	 1  & 0
	\end{array}
\right]
\]
Then a straightforward calculation shows that
\begin{eqnarray*}
g(x_1,x_2,x_3) &=& z_3^{-4}[x_2^2dx_1^2 + 2x_3^2 dx_1dx_2 - 4x_2x_3dx_1dx_3  \\ & & + x_1^2dx_2^2 - 4x_1x_3 dx_2dx_3 + 2(x_1x_2+x_3^2)dx_3^2].
\end{eqnarray*}
The proposition follows by applying the pull-back $\Theta^*$ to this form. 
\end{proof}

\begin{corollary}
The volume element in the $z_j$ coordinates is 
\begin{equation}\label{eq:volume}
dV = \frac{4\, dz_1\,dz_2\,dz_3}{(1-z_1^2-z_2^2)^{2}\,z_3}
\end{equation}
\end{corollary}

\begin{corollary} \label{cor:laplace-beltrami}
The Laplace -Beltrami operator in $\DR$ in $z_j$ coordinates is
\begin{equation}\label{eq:Laplace}
\triangle = \frac{(1-z_1^2-z_2^2)^2}{4}\left( \frac{\partial^2}{\partial z_1^2}+\frac{\partial^2}{\partial z_2^2}\right) + z_3\frac{\partial}{\partial z_3} + z_3^2\frac{\partial^2}{\partial z_3^2} 
\end{equation}
We note $\triangle_D$ the first term on the r.h.s.
\end{corollary}

Let us now compute the distance $d$ in $\DR$. We set $\bz=(z_1,z_2,z_3)$ and $\bz'=(z'_1,z'_2,z'_3)$. Then
\begin{proposition}
\begin{equation}\label{eq:distance}
d(\bz,\bz') = \sqrt{(d_2((z_1,z_2),(z'_1,z'_2))^2 + \log^2(\frac{z_3}{z'_3})}
\end{equation}
where $d_2$ denotes the distance in the Poincar\'e disc. 
\end{proposition}
\begin{proof}
We start from expression (\ref{distanced0}) which gives the distance between two tensors $\T_1$ and $\T_2$.  Let $\T=z_3\tilde T$, $\T'=z'_3\tilde T'$ and $\alpha=z_3/z'_3$. We note $\lambda$, $1/\lambda$ the eigenvalues of $\tilde T^{-1}\tilde T'$. Then
\[
 d_0(\T,\T')^2=\log^2 \alpha\lambda+\log^2 \frac{\alpha}{ \lambda}=2\log^2 \alpha+2\log^2 \lambda=2\log^2 \alpha+d_0(\tT,\tT')^2.
\]
But $d_0(\tT,\tT')=d_2((z_1,z_2),(z'_1,z'_2))$, hence the result.
\end{proof}

Note that the distance in $D$ is given by the following formula where $z$, $z'$ are complex numbers in the unit disc.
\begin{equation} \label{d2}
d_2(z,z')=2\,\arctanh \frac{|z-z'|}{|1-\overline zz'|},
\end{equation}

\subsection{Periodic lattices in $\D\times \R^{+}_{*}$} \label{section:lattices}

Let us recall first some basic properties of the group of isometries in the hyperbolic plane (which we shall always identify with the Poincar\'e disc $\D$ in the sequel).  We refer to textbooks in hyperbolic geometry for details and proofs. 
The direct isometries (preserving the orientation) in $\D$ are the elements of the special unitary group, noted ${\rm SU}(1,1)$, of $2\times 2$ Hermitian matrices with determinant equal to 1. Given
\[
\gamma =\left[
\begin{array}{cc}
\alpha & \beta \\ \overline\beta & \overline\alpha \end{array}
\right] \ \text{such that}\ |\alpha|^2-|\beta|^2=1,
\]
the corresponding isometry $\gamma$ is defined by
\begin{equation} \label{eq:motionD}
\gamma \cdot z = \frac{\alpha z+\beta}{\overline\beta z+\overline\alpha},~~z\in \D
\end{equation}
Orientation reversing isometries of $\D$ are obtained by composing any transformation (\ref{eq:motionD}) with the reflection $\kappa:~z\mapsto \overline z$. The full symmetry group of the Poincar\'e disc is therefore
$${\rm U}(1,1)={\rm SU}(1,1)\cup \kappa\cdot {\rm SU}(1,1).$$

These isometries preserve angles, however they do not transform straight lines into straight lines. Given two points $z\neq z'$ in $\D$, there is a unique geodesic passing through them: the portion in $\D$ of the circle containing $z$ and $z'$ and intersecting the unit circle at right angles. This circle degenerates to a straight line when the two points lie on the same diameter. Any geodesic uniquely defines the reflection through it. Reflections are orientation reversing, one representative being the complex conjugation: $\kappa \cdot z=\overline{z}$. 

We distinguish three different kinds of direct isometries in $\D$, according to which conjugacy class of the following one parameter subgroups it does belong:
\begin{definition} \label{def_KAN}
\[
\left\{
\begin{array}{lcl}
K & = & \{r_\varphi =\left[\begin{array}{cc} e^{i\varphi/2} & 0 \\ 0 & e^{-i\varphi/2} \end{array} \right], \quad \varphi\in S^1\}\\
&&\\
A & = & \{a_t=\left[\begin{array}{cc} \cosh t/2 & \sinh t/2 \\ \sinh t/2 & \cosh t/2 \end{array} \right], \quad t \in \R\}\\
&&\\
N & = & \{n_s=\left[\begin{array}{cc} 1+is & -is \\ is & 1-is \end{array} \right], \quad s \in \R\}
\end{array}
\right.
\]
\end{definition}
Note that $r_\varphi\cdot z=e^{i\varphi}\, z$ for $z\in \D$ and also, $a_t\cdot 0=\tanh(t)$. The elements of $A$ are sometimes called ``boosts''  \cite{balazs-voros:86}.
The following theorem is called the Iwasawa decomposition (see \cite{iwaniec:02}).
\begin{theorem}
\[
{\rm SU}(1,1)=K\cdot A\cdot N 
\]
\end{theorem}
\noindent 
The orbits of $A$ converge to the same limit points $b_{\pm 1}=\pm 1$ on the unit circle when $t\rightarrow\pm\infty$. In particular the diameter $(b_{-1},b_{1})$ is the orbit $\{x=\tanh(t),~t\in\R\}$. The orbits of $N$ are the circles tangent to the unit circle at $b_1$. These circles are called {\em horocycles} with limit point $b_1$ ($N$ is called the horocyclic group). The orbits of $K$ are circles inside $\D$ (they coincide with Euclidean circles only when they are centered at the origin). \\
Isometries of the types $K$, $A$, $N$ are respectively called elliptic, hyperbolic and parabolic. Elliptic isometries have one fixed point in $\D$ while hyperbolic isometries have two fixed points on the boundary $\partial \D$ and parabolic isometries have one fixed point on $\partial \D$ (hence "at infinity"). Hyperbolic isometries play the same role as translations in Euclidean space while parabolic isometries have no counterpart. There is however one important difference with Euclidean translations: two hyperbolic translations do not commute in general. The fundamental difference from this point of view between $\D$ and the Euclidean plane $\R^2$ is that the latter is itself an Abelian group while $\D\simeq SU(1,1)/SO(2,\R)$ is not a group. This makes its analysis, especially its Fourier analysis, harder and less intuitive. 

We can now define a (periodic) lattice in $\D$ and in $\D\times \R^*$. Let $\Gamma$ be a discrete subgroup of $SU(1,1)$ such that the orbits of points in $\D$ under the action of $\Gamma$ have no accumulation point in $\D$. This is a Fuchsian group. To any Fuchsian group we can associate a {\em fundamental domain} which is the closure, noted $F_\Gamma$, of an open set $\overset{o}{F}_\Gamma\subset \D$ with the following properties  \cite{katok:92}:  
\begin{itemize}
\item[(i)] if $\gamma\neq Id\in\Gamma$, then $\gamma F_\Gamma\cap \overset{o}{F}_\Gamma = \emptyset$;
\item[(ii)] $\underset{\gamma\in\Gamma}{\bigcup}\, \gamma F_\Gamma =\D$.
\end{itemize}
Hence $F_\Gamma$ generates a periodic tiling (or tessellation) of $\D$. A fundamental domain need not be a compact subset of $\D$ (it may have vertices on the circle at infinity $\partial\D$). When it does, $\Gamma$ is called a {\em cocompact Fuchsian group}. In this case $\Gamma$ contains no parabolic element, its area is finite and a fundamental domain can always be built as a polygon (the Dirichlet region, see \cite{katok:92}). As mentioned in the introduction we identified in \cite{chossat-faugeras:09} a family of Fuchsian groups, some of them compact, that naturally arose from the analysis of the symmetries in the spatial distribution of the photoreceptors in the retina.

The following definition is just a translation to the hyperbolic plane of the definition of an Euclidean lattice.
\begin{definition}
A lattice group of $\D$ is a cocompact Fuchsian group which contains no elliptic element.
\end{definition}
The action of a lattice group has no fixed point, therefore the quotient surface $\D/\Gamma$ is a (compact) manifold and it is in fact a Riemann surface. A remarkable theorem states that any compact Riemann surface is isomorphic to a lattice fundamental domain of $\D$ if and only if it has genus $g\geq 2$ \cite{katok:92}. The case $g=1$ corresponds to lattices in the Euclidean plane (in this case there are three kinds of fundamental domains: rectangles, squares and hexagons). The simplest lattice in $\D$, with genus 2, is generated by an octagon and will be studied in detail in Section \ref{section:octagon}. 

Given a lattice, we may ask what is the symmetry group of the fundamental domain $F_\Gamma$, identified with the quotient surface $\D/\Gamma$. Indeed, this information will play a fundamental role in the subsequent bifurcation analysis. In the case of Euclidean lattice, the quotient $\R^2/\Gamma$ is a torus $\cal{T}$ (genus one surface), and the group of automorphisms is $\cal{H}\rtimes \cal{T}$ where $\cal{H}$ is the holohedry of the lattice: ${\cal{H}}=D_2, D_4$ or $D_6$ for the rectangle, square and hexagonal lattices respectively. In the hyperbolic case the group of automorphisms of the surface is finite. In order to build this group we need first to introduce some additional definitions.

Tilings of the hyperbolic plane can be generated by reflections through the edges of a triangle $\tau$ with vertices $P$, $Q$, $R$ and angles $\pi/\ell$, $\pi/m$ and $\pi/n$ respectively, where $\ell,~m,~n$ are integers such that $1/\ell+1/m+1/n<1$ \cite{katok:92}. 

\begin{figure}[htp]
\centering \includegraphics[width=0.5\textwidth]{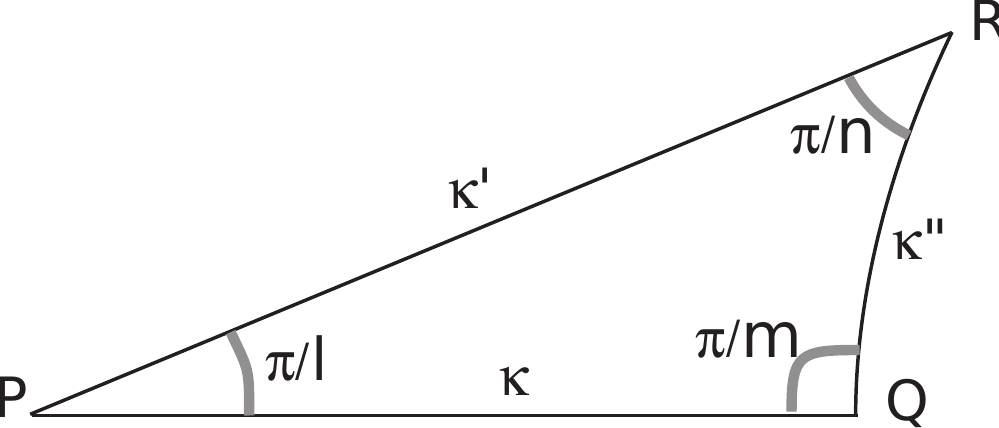}
\caption{{\bf The triangle $\tau$, also noted $T(2,3,8)$. The values of $l$, $m$, and $n$ are $l=8$, $m=2$ and $n=3$. }}
\label{fig:triangle}
\end{figure}

\begin{figure}
\centering
\includegraphics[width=0.8\textwidth]{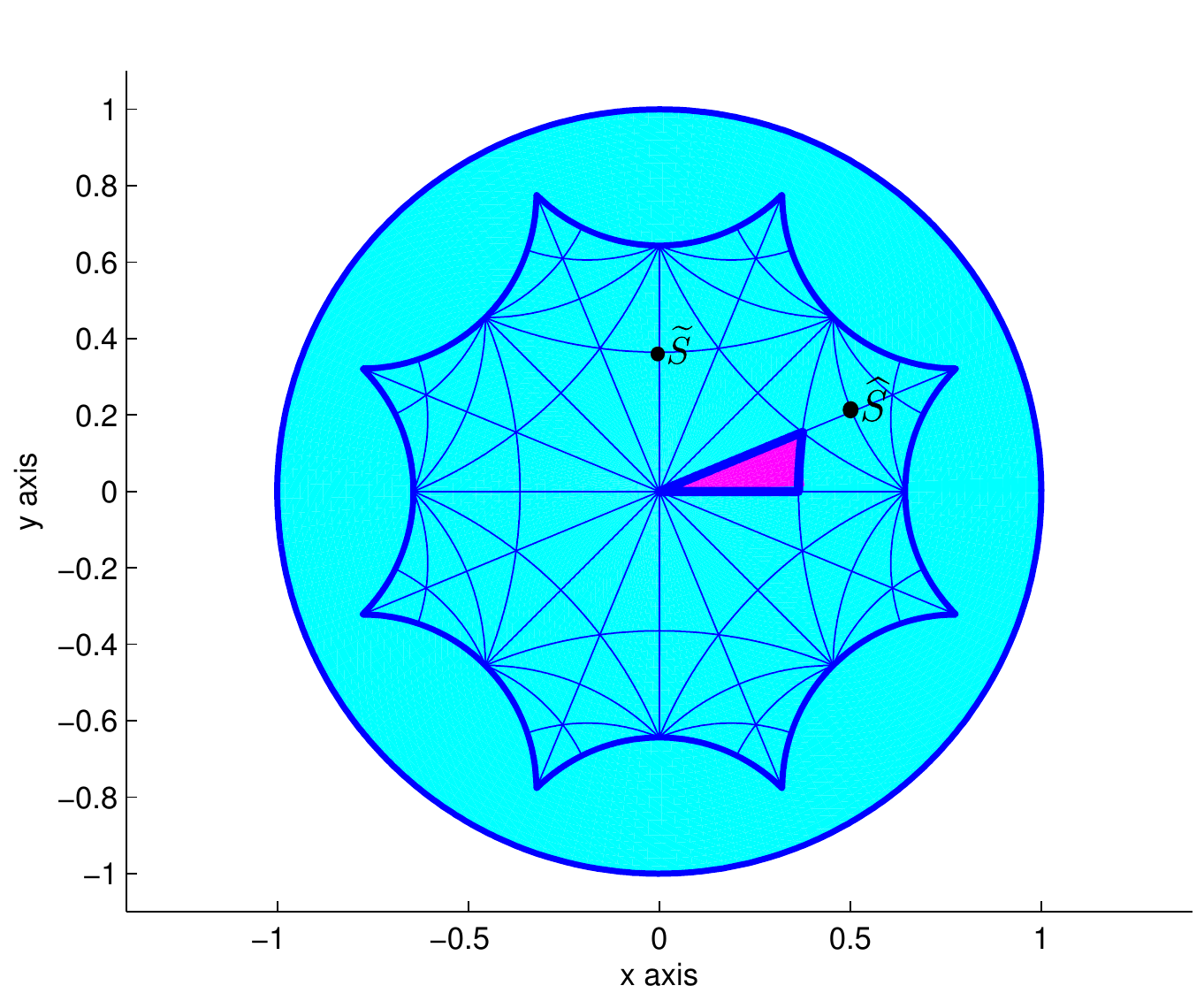}
\caption{Tesselation of the hyperbolic octagon with the triangle $T(2,3,8)$, colored in purple in the plot. We define two points $\widehat{S}$ and $\widetilde{S}$.  $\widehat{S}$ is the center of the rotation $\hat \sigma$ by $\pi$ (mod $\Gamma$), see text in subsection \ref{subsection:octlattice}. $\widetilde{S}$ is the center of the rotation $\tilde{\sigma}$ by $\pi$ (mod $\Gamma$), see text in subsection \ref{subsection:octlattice}.}
\label{fig:tesselation}
\end{figure}

Remember that reflections are orientation-reversing isometries. We note $\kappa$, $\kappa'$ and $\kappa''$ the reflections thro
ugh the edges $PQ$, $QR$ and $RP$  respectively (Figure \ref{fig:triangle}). The group generated by these reflections contains 
an index 2 Fuchsian subgroup $\Lambda$ called a triangle group, which always contains elliptic elements because the product of 
the reflections through two adjacent edges of a polygon is elliptic with fixed point at the corresponding vertex. One easily shows that $\Lambda$ is generated by the rotations of angles $2\pi/l$, $2\pi/m$ and $2\pi/n$ around the vertices $P$, $Q$, $R$ re
spectively. A fundamental domain of $\Lambda$ is the "quadrangle" $F_\Lambda =\tau\cup\kappa\tau$ \cite{katok:92}. Note that $F
_\Lambda\simeq \D/\Lambda$ is a sphere (genus 0 surface) obtained by identifying the three edges of $\tau$.
The subgroup of hyperbolic translations in $\Lambda$ is a lattice group $\Gamma$, normal in $\Lambda$, whose fundamental domain is filled with copies of the basic tile $\tau$. The group of orientation-preserving automorphisms of $F_\Gamma \simeq \D / \Gamma$ is therefore $G=\Lambda/\Gamma$. From the algebraic point of view, $G$ is generated by three elements $a,~b,~c$ satisfy therelations $a^\ell=b^m=c^n=1$ and $a\cdot b\cdot c=1$. We say that $G$ is an $(l,m,n)$ group.  Taking account of orientation-reversing isometries, the full symmetry group of $F_\Gamma$ is $G^*=G\cup \kappa G=G\rtimes \Z_2(\kappa)$. This is also a tiling group of $F_\Gamma$ with tile $\tau$: the orbit $G^*\tau$ fills $F_\Gamma$ and its elements can only intersect at their edges. 

Given a lattice, how to determine the groups $G$ and $G^*$? The following theorem gives conditions for this, see \cite{broughton-dirks-etal:01}. 
\begin{theorem}
An $(l,m,n)$ group $G$ is the tiling rotation group of a compact Riemann surface of genus $g$ if and only if its order satisfies the Riemann-Hurwitz relation
$$ |G|=\frac{2g-2}{1 - (\frac{1}{\ell}+\frac{1}{m}+\frac{1}{n})}.$$
\end{theorem}
Tables of triangle groups for surfaces of genus up to 13 can be found in \cite{broughton-dirks-etal:01}.

These definitions extend naturally to $SDP(2) \simeq\D\times\R^*$ as follows.

\begin{definition}
A lattice group of $\D\times\R^*$ is a subgroup of the form $\Gamma\times\Xi$ where $\Gamma$ is a lattice group acting in $\D
$ and $\Xi$ is a non trivial discrete subgroup of $\R^*$.
\end{definition}
Any discrete subgroup of $\R^*$ is generated by a positive number $a$ and can be further identified with $\Z$: $\Xi=\{a^n,~n\in\Z\}$. A fundamental domain for  $\Gamma\times\Xi$ is a "box" $F_\Gamma\times [1,a]$. 

\subsection{Plane waves in SDP(2)}

Let us first recall the Euclidean setting for functions defined in $\R^3$ (we could take any other $\R^n$, $n>0$). In this case every function of the form $e^{\lambda{\bf k}\cdot\mr}$ where ${\bf k}\in\R^3$ is a unit vector, is an eigenfunction of the Laplace operator in $\R^3$: 
\[
\Delta e^{\lambda{\bf k}\cdot\mr}=-\lambda^2e^{\lambda{\bf k}\cdot\mr},\,\mr \in \R^3.
\]
The fact that the eigenvalues do not depend upon the direction of the wave vector ${\bf k}$ reflects the rotational invariance of the Laplace operator. Moreover if we take $\lambda=i\alpha$, $\alpha\in\R$, then $e^{\lambda\mk \cdot\mr}$ is invariant under translations in $\R^2$ by any vector $\me$ satisfying the condition $\mk \cdot\me=2n\pi$ where $n\in\Z$ (it clearly does not depend upon the coordinate along the axis orthogonal to $\mk$). The functions $e^{i\alpha\mk \cdot\mr}$ are elementary spatial waves in $\R^3$. \\
Now, given $\alpha >0$ and a basis of unit vectors $\{\mk_1,\mk_2,\mk_3\}$ of $\R^3$ we can define the translation group ${\cal L}$ spanned by ${\bf e}_i$, $i=1,2,3$, such that $\mk_i\cdot {\bf e}_j=2\pi/\alpha\delta_{ij}$. Hence ${\cal L}$ is a lattice group of $\R^3$. It defines a periodic tiling, the fundamental domain of which is a compact cell which we may identify with the quotient space $\R^3/{\cal L}$ and which  we can identify with a 3-torus.  Any smooth enough function in $\R^3$ which is invariant under the action of ${\cal L}$ can be expanded in a Fourier series of elementary spatial waves $e^{i\alpha(m\mk_1+n\mk_2+p\mk_3)\cdot\mr}$, $m,n,p\in\Z$. The Laplace operator in the space of square-integrable functions in $\R^3/{\cal L}$ is self-adjoint and its spectrum consists of real isolated eigenvalues with finite multiplicities. The multiplicity depends upon the {\em holohedry} of the lattice, which we defined in the previous section (the largest subgroup of $O(3)$ leaving invariant the lattice). There are finitely many holohedries (see \cite{miller:72} for details). It follows from the above considerations that by restricting the analysis to classes of functions which are invariant under the action of a lattice group, one can apply standard techniques of equivariant bifurcation theory to assert the generic existence of branches of solutions of Euclidean invariant bifurcation problems, which are spatially periodic with respect to lattice groups and whose properties are largely determined by the holohedry of the lattice \cite{golubitsky-stewart-etal:88},\cite{dionne-golubitsky:92}.  
Note also this was the approach of  \cite{bressloff-cowan-etal:02} for the analysis of the occurence of visual hallucinations in the cortex.

We wish to apply the same idea to bifurcation problems defined in $SDP(2)$. For this we need to define elementary eigenfunctions of the Laplace-Beltrami operator such that spatially periodic functions (in a sense to be defined later) can be expanded in series of these elementary "waves" in $SDP(2)$. In the sequel $\triangle$ will denote the Laplace-Beltrami operator in $SDP(2)$ or, equivalently, in $\DR$.

Let $b$ be a point on the circle $\partial D$, which we may  take equal to $b_1=1$  after a suitable rotation. For $z\in \D$, we define the "inner product" $\langle z,b \rangle$ as the algebraic distance to the origin of the (unique) horocycle based at $b$ and passing through $z$. This distance is defined as the hyperbolic signed length of the segment $O\xi$ where $\xi$ is the intersection point of the horocycle and the line (geodesic) $Ob$. Note that $\langle z,b \rangle$ does not depend on the position of $z$ on the horocycle. In other words, $\langle z,b \rangle$ is invariant under the action of the one-parameter group $N$ (see definition above). The "hyperbolic plane waves" 
\[
e_{\rho,b}(z)=e^{(i\rho+\frac{1}{2})\langle z,b \rangle}, \,\rho\in\C,
\]
satisfy
\[
-\triangle_D\, e_{\rho,b} = (\rho^2+\frac{1}{4}) e_{\rho,b}.
\]
where $\triangle_D$ is defined in Corollary \ref{cor:laplace-beltrami}. These are the elementary eigenfunctions with which Helgason built a Fourier transform theory for the Poincar\'e disc, see \cite{helgason:00}. It follows from Helgason's theory that any eigenfunction of $\triangle_\D$ can be expressed as an integral over the boundary elements:

\begin{theorem} \label{theorem: Helgason} \cite{helgason:00}
Any eigenfunction of the operator $-\triangle_\D$ admits a decomposition of the form
\[
\int_{\partial D}{e_{\rho,b}(z) dT_\rho(b)}
\]
where $T_\rho$ is a distribution defined on the circle $\partial D$ and the eigenvalue is $\rho^2+\frac{1}{4}$.
\end{theorem}
Real eigenvalues  $-(\rho^2+\frac{1}{4})$ of $\triangle_\D$ correspond to taking $\rho$ real or $\rho \in i\R$. The latter case is irrelevant for our study as it corresponds to exponentially diverging eigenfunctions. Therefore the real spectrum of $\triangle$ is continuous and is bounded from above by $-1/4$. By using (\ref{eq:Laplace}) we extend Theorem \ref{theorem: Helgason} to $\DR$:

\begin{corollary} \label{corollaire: fonctions propres}
(i) Let us note $\bz=(z_1,z_2,z_3)\in \DR$ and $z=z_1+iz_2$. The function 
\begin{equation} \label{elementaryeigenfunction}
\psi_{\rho,b,\beta}(\bz) = e_{\rho,b}(z)\,e^{i\log \beta \log z_3}
\end{equation}
satisfies the relation $\triangle \psi_{\rho,b,\beta} = -(\rho^2+\frac{1}{4}+\log^2 \beta) \psi_{\rho,b,\beta}$. \\
(ii) Any eigenfunction of $\triangle$ admits a decomposition of the form
\begin{equation} \label{eigenfunction}
e^{i\log \beta\log z_3} \int_{\partial \D}{e_{\rho,b}(z)\,dT_\rho(b)}.
\end{equation}
\end{corollary}
From there we can extend the definition (and properties) of the Fourier transform in the hyperbolic plane given by Helgason \cite{helgason:00} to the space $\DR$:
\begin{definition}
Given a function $f$ on $\DR$, its Fourier transform is defined by
\begin{equation} \label{H-Fourier}
\tilde f(\rho, b,\beta) =  \int_{\DR}{f(\bz)e_{-\rho,b}(z)e^{-i \log \beta \log z_3} d\bz}
\end{equation}
\end{definition}

In the following we will look for solutions of bifurcation problems in $\DR$, which are invariant under the action of a lattice group: $(\gamma,\xi)\cdot u(z,z_3)=u(\gamma^{-1}z,\xi^{-1}z_3)=u(z,z_3)$ for $\gamma\in\Gamma$, $\xi\in\Xi$. This boils down to looking for the problem restricted to a fundamental domain with suitable boundary conditions imposed by the $\Gamma$-periodicity, or, equivalently, to looking for the solutions of the problem projected onto the orbit space $\D/\Gamma\times \R^+_*/\Xi$ (which inherits a Riemannian structure from $\DR$). Because the fundamental domain is compact, it follows from general spectral theory that $-\triangle$ is self-adjoint, non negative and has compact resolvent in $L^2(\D/\Gamma\times \R^+_*/\Xi)$ \cite{buser:92}. Hence its spectrum consists of real positive and isolated eigenvalues of finite multiplicity.  

Coming back to Theorem \ref{theorem: Helgason}, we observe that those eigenvalues $\lambda$ of $-\triangle_\D$ which correspond to  $\Gamma$-invariant eigenfunctions, must have $\rho\in\R$ or $\rho\in i\R$. The case $\rho$ real corresponds to the Euclidean situation of planar waves with a given wave number, the role of which is played by $\rho$ in $\D$. In this case the eigenvalues of $-\triangle_\D$ satisfy $1/4< \lambda$. On the other hand there is no Euclidean equivalent of the case $\rho\in i\R$, for which the eigenvalues $0 < \lambda \leq 1/4$ are in finite number. It turns out that such "exceptional" eigenvalues do not occur for "simple" groups such as the octagonal group to be considered in more details in the Section \ref{section:octagon}. This follows from formulas which give lower bounds for these eigenvalues. Let us give two examples of such estimates (derived by Buser \cite{buser:92}, see also \cite{iwaniec:02}): (i) if $g$ is the genus of the surface $\D/\Gamma$, there are at most $3g-2$ exceptional eigenvalues; (ii) if $d$ is the diameter of the fundamental domain, then the smallest (non zero) eigenvalue is bounded from below by $\left( 4\pi\, \sinh \frac{d}{2}\right)^{-2}$. 

Suppose now that the eigenfunction in Theorem \ref{theorem: Helgason} is $\Gamma$-periodic. Then the distribution $T_\rho$ satisfies  the following equivariance relation \cite{pollicott:89}. Let $\gamma (\theta)$ denote the image of $\theta\in\partial \D$ under the action of $\gamma\in\Gamma$. Then 
$$T_\rho(\gamma\cdot\theta)=|\gamma'(\theta)|^{\frac{1}{2}+i\rho}\,T_\rho(\theta).$$ 
As observed by \cite{series:87}, this condition is not compatible with $T_\rho$ being a "nice" function. In fact, not only does there not exist any explicit formula for these eigenfunctions, but their approximate computation is itself an uneasy task. We shall come back to this point in subsequent sections.

\section{Bifurcation of patterns in SDP(2)}\label{section:bifurcation}

We now consider again equation (\ref{eq:neuralmass}), which we set in $\DR$ by the change of coordinates (\ref{diffeoTheta}). Assuming $I=0$ (no external input)  and the invariance hypothesis (\ref{Invariance hypothesis}) for the connectivity function $w$, and after a choice of time scale such that $\alpha=1$, the equation reads
\begin{equation}\label{eq: I=0}
\frac{\partial V}{\partial \tau} = -V + \mu w \ast V + R(V)
\end{equation} 
where:
\begin{itemize}
\item $\mu=S'(0)$,
\item $w \ast V$ denotes the convolution product $\int_{\DR}{w(\bz,\bz')V(\bz')\,d\bz'}$ (with $w(\bz,\bz')=f(d(\bz,\bz'))$),
\item $R(V)$ stands for the remainder terms in the integral part of (\ref{Invariance hypothesis}). This implies $R'(0)=0$.
\end{itemize}
It is further assumed that $f$ is a "Mexican hat" function, typically of the form
$$
f(x)=\frac{1}{\sqrt{2\pi \sigma_1^2}}  e^{-\frac{x^2}{2\sigma_1^2}}-\theta \frac{1}{\sqrt{2\pi \sigma_2^2}} e^{-\frac{x^2}{2\sigma_2^2}},
$$
where $\sigma_1 < \sigma_2$ and $\theta \leq 1$. 

Let us look at the linear stability of the trivial solution of (\ref{eq: I=0}) against perturbations in the form of hyperbolic waves (\ref{elementaryeigenfunction}) with $\rho\in\R$. This comes back to looking for $\sigma$'s such that
$$
\sigma = -1 + \mu\hat w
$$
where $\hat w$ is the hyperbolic Fourier transform of $w$ as defined in Definition \ref{H-Fourier}. The numerical calculation shows that for each value of $\rho$ and $\beta$, there exists a value $\mu(\rho,\beta)$ such that if $\mu<\mu(\rho,\beta)$ then all $\sigma$'s are negative, while $\sigma=0$ at $\mu=\mu(\rho,\beta)$. The "neutral stability surface" defined by $\mu(\rho,\beta)$ is typically convex and reaches a minimum $\mu_c$ at some values $\rho_c,\beta_c$. Therefore when $\mu<\mu_c$ the trivial state $V=0$ is stable against such perturbations while it becomes marginally stable when $\mu=\mu_c$ with critical modes $\psi_{\rho_c,b,\beta_c}$, for any $b\in S^1$ (rotational invariance). Therefore a bifurcation takes place at this critical value.  

The situation is absolutely similar if instead of equation \ref{eq: I=0} we consider systems of PDEs in $\DR$ with pattern selection behavior and with $U(1,1)\times R^+_*$ invariance. A paradigm for such systems is the "Laplace-Beltrami" version of Swift-Hohenberg equation
$$
\frac{\partial u}{\partial t} = \mu u - (\triangle +\alpha)^2 u + u^2,~~\alpha\in\R^+_*.
$$
with $\triangle$ as in (\ref{eq:Laplace}).

It is unconceivable to solve the bifurcation problem at this level of generality, because the fact that the spectrum is continuous plus that each eigenvalue $\sigma$ has an infinite multiplicity (indifference to $b$) makes impossible the use of the classical tools of bifurcation theory. As in the Euclidean case of pattern formation, we therefore want to look for solutions in the restricted class of patterns which are spatially periodic. In the present framework, this means looking for bifurcating patterns which are invariant under the action of a lattice group $\Gamma\times\Z^+_*$ in $SU(1,1)\times\R^+_*$. There is however an immediate big difference with the Euclidean case. While in the latter any critical wave number $\alpha_c$ can be associated with a periodic lattice (of period $2\pi/\alpha_c$), in the hyperbolic case not every value of $\rho_c$ can be associated with a lattice in $\D$. More precisely it is not generic for the Laplace-Beltrami operator in $\D$ to possess eigenfunctions, as defined in Theorem \ref{theorem: Helgason}, which are invariant under a lattice group. We can therefore look for the bifurcation of spatially periodic solutions associated with a given lattice, but these patterns will not in general correspond to the most unstable perturbations unless the parameters in the equations are tuned so that it happens this way. The question of the observability of such patterns is therefore completely open. 

We henceforth look for patterns in $\DR$ which are invariant under a lattice $\Gamma$ in $\D$ and which are periodic, with period $2\pi/\beta_c$, in $\R^+_*$. This boils down to looking for solutions in the space $L^2(\D/\Gamma\times\R^+_*/\beta_c\Z^+_*)$. Note that $\R^+_*/\beta_c\Z^+_*\simeq S^1$. With a suitable inner product this space admits an orthonormal Hilbert basis which is made of functions of the form 
$$
\Psi(z)e^{ni\log(\beta_c) \log(z_3)},~~z\in\D,~~n\in \N
$$ 
where $\Psi$ are the eigenfunctions of $\triangle$ in $L^2(\D/\Gamma)$. As we mentionned in the previous section, these eigenfunctions are not known explicitely. By restricting the "neutral stability surface" $\mu(\rho,\beta)$ to those values which correspond to eigenfunctions with $\Gamma\times\beta_c\Z^+_*$ periodicity, we obtain a discrete set of points on this surface with one minimum $\mu_0$ associated with a value $\rho_0$ of $\rho$ and $\beta_0$ of $\beta$. In general this minimum is unique. Moreover the multiplicity of the 0 eigenvalue is now finite and this eigenvalue is semi-simple. Let us call $X$ the eigenspace associated with the 0 eigenvalue (therefore $X$ is the kernel of the critical linear operator).   

The full symmetry group of $\D/\Gamma\times S^1$ is equal to $G^*\times O(2)$ where $G^*$ is the (finite) group of automorphisms in $U(1,1)$ of the Riemann surface $\D/\Gamma$ (notation of Section \ref{section:lattices}) and $O(2)$ is the symmetry group of the circle (generated by $S^1$ and by reflection across a diameter). The equation restricted to this class of $\Gamma\times \beta_c\Z^+_*$-periodic patterns is invariant under the action of  $G^*\times O(2)$. 
We can therefore apply an equivariant Lyapunov-Schmidt reduction to this bifurcation problem \cite{chossat-lauterbach:00}, leading to a bifurcation equation in $X$
\begin{equation} \label{eq:bifurcation}
f(x,\mu)=0,~~x\in X
\end{equation}
where $f~:~X\times\R\rightarrow X$ is smooth, $f(0,0)=0$, $\partial_xf(0,0)$ is not invertible and $f(\cdot,\mu)$ commutes with the action of $G^*\times O(2)$ in $X$. 

Now the methods of equivariant bifurcation theory can be applied to (\ref{eq:bifurcation}). In particular we can apply the Equivariant Branching lemma (see \cite{golubitsky-stewart-etal:88} for a detailed exposition):

\begin{theorem} \label{lemme:equivariantbranching}
Suppose the action of $G^*\times O(2)$ is absolutely irreducible in $X$ (i.e. real equivariant linear maps in $X$ are scalar multiple of the identity). Let $H$ be an isotropy subgroup of $G^*\times O(2)$ such that the subspace $X^H=\{x\in X~|~H\cdot x=x\}$ is one dimensional. Then generically a branch of solutions of (\ref{eq:bifurcation}) bifurcates in $X^H$. The conjugacy class of $H$ (or isotropy type) is called "symmetry breaking".
\end{theorem}
\noindent Let us briefly recall the meaning of this theorem. By equivariance of $f$, any subspace of $X$ defined as $X^H$, $H$ a (closed) subgroup of $G^*\times O(2)$, is invariant under $f$. By the irreducibility assumption, if $H=G^*\times O(2)$, then $\{x\in X~|~H\cdot x=x\}=\{0\}$. Therefore $f(0,\mu)=0$ for all $\mu$. Now the assumption of absolute irreducibility implies that $\partial_xf(0,\mu)= a(\mu)Id_X$ where $a$ is a smooth real function such that $a(0)=0$ and generically $a'(0)\neq 0$. It follows that if now $H$ is a subgroup such that $\dim X^H=1$, then equation (\ref{eq:bifurcation}) restricted to this subspace reduces to a scalar equation $0=a'(0)\mu x + xh(x,\mu)$ (with $h(0,0)=0$), which has a branch of non trivial solutions by the implicit function theorem. 

\noindent {\em Remark 1:} the word "generically" can be interpreted as follows: the result will fail only if additional degeneracies are introduced in the equations. See \cite{chossat-lauterbach:00}  and \cite{golubitsky-stewart-etal:88} for a rigorous definition and proof. 

\noindent {\em Remark 2:} the assumption of absolute irreducibility is itself generic (in the above stated sense) for one-parameter steady-state bifurcation problems. A given irreducible representation of a compact or finite group need not be absolutely irreducible, this fact has to be proven. 

\noindent {\em Remark 3:} Theorem \ref{lemme:equivariantbranching} does not necessarily give an account of all possible branches of solutions of (\ref{eq:bifurcation}), see \cite{chossat-lauterbach:00}. It gives nevertheless a large set of generic ones. To go further it is necessary to compute the equivariant structure of $f$, or at least of its Taylor expansion to a sufficient order. The same is true if one wants to determine the stability of the bifurcated solutions, within the class of $\Gamma\times S^1$ periodic solutions of the initial evolution equation.

This theorem, together with the knowledge of the lattices and the (absolutely) irreducible representations of the groups $G^*$, gives us a mean to classify the periodic patterns which can occur in $\D\times \R^+_*$.
By analogy with the Euclidean case (bifurcation of spatially periodic solutions in the Euclidean space), we call {\em H-planforms} the solutions of a $U(1,1)$ (resp. $GL(2\R)$) invariant bifurcation problem in $\D$ (resp. $\DR$), which are invariant by a lattice group $\Gamma$ (resp. $\Gamma\times S^1$). 

Being interested here in the classification of solutions rather than in their actual computation for a specific equation, all remains to do is to determine the absolutely real irreducible representations of the group $G^*\times O(2)$  and the computation of the dimensions of the subspaces $X^H)$. For this purpose we can get rid of the $S^1$ component of the domain of periodicity. Indeed let $H=H_1\times H_2$ be an isotropy subgroup for the representation $R$ of $G^*\times O(2)$ acting in $X$. Then $X=V\otimes W$ and $R=S\otimes T$, where $S$ is an irreducible representation $G^*$ in $V$  and $T$ is an irreducible representation of $O(2)$ in $W$ \cite{serre:78}, and therefore $H_1$ acts in $V$ and $H_2$ acts in $W$. Now we have the following lemma, the proof of which is straightforward:

\begin{lemma} \label{lemme:equivariantbif}
${\rm dim}(X^H)=1$ if and only if  ${\rm dim}(V^{H_1})={\rm dim}(W^{H_2})=1$.
\end{lemma}

Now, the irreducible real representations of $O(2)$ are well-known: they are either one dimensional (in which case every point is rotationally invariant) or two-dimensional, and in the latter case the only possible one dimensional subspaces $W^{H_2}$ are the reflection symmetry axes in $\R^2$ (which are all equivalent under rotations in $O(2)$). It follows that the classification is essentially obtained from the classification of the isotropy subgroups $H_1$ of the irreducible representations of $G^*$.

Once the irreducible representations are known, this can be achieved by applying the "trace formula" \cite{golubitsky-stewart-etal:88,chossat-lauterbach:00}:
\begin{proposition}
Let $H$ be a subgroup of $G^*$ acting in a space $V$ by a representation $\rho:~G^*\rightarrow Aut(V)$, then 
\begin{equation} \label{eq:formuletrace}
{\rm dim}(V^{H})=\frac{1}{|H|} \sum_{h\in H}{\rm tr}(\rho(h))
\end{equation}
\end{proposition}
Note that ${\rm tr}(\rho)$ is the character of the representation $\rho$ (a homomorphism $G^*\rightarrow \C$). What is really needed to apply the Equivaraint Branching Lemma is therefore the character table of the representations.
In the next section we investigate this classification in the case when the lattice is the regular octogonal group.

\section{A case study: the octagonal lattice}\label{section:octagon}
According to the comment following Lemma \ref{lemme:equivariantbif}, in all of this section we only consider the classification of isotropy subgroups satisfying the conditions of the equivaraint branching lemma in the Poincar\'e disc $\D$.

\subsection{The octagonal lattice and its symmetries}\label{subsection:octlattice}
Among all lattices in the hyperbolic plane, the octagonal lattice is the simplest one. As before we use the Poincar\'e disc representation of the hyperbolic plane. Then the octagonal lattice group $\Gamma$ is generated by the following four hyperbolic translations (boosts), see \cite{balazs-voros:86}:

\begin{equation}
g_0 = \left(\begin{array}{cc}1+\sqrt{2} & \sqrt{2+2\sqrt{2}} \\ \sqrt{2+2\sqrt{2}} & 1+\sqrt{2}\end{array}\right)
\end{equation}
and $g_j = r_{j\pi/4}g_0r_{-j\pi/4}$, $j=1,2,3$, where $r_\varphi$ indicates the rotation of angle $\varphi$ around the origin in $\D$. The fundamental domain of the lattice is a regular octagon $\Oct$ as shown in Figure. The opposite sides of the octagon are identified by periodicity, so that the corresponding quotient surface $\D/\Gamma$ is isomorphic to a "double doughnut" (genus two surface) \cite{balazs-voros:86}. Note that the same octagon is also the fundamental domain of another group, not isomorphic to $\Gamma$, obtained by identifying not the opposite sides but pairs of sides as indicated in Figure. This is called the Gutzwiller octagon. A procedure of classification of the lattices using graphs is presented in \cite{sausset-tarjus:07}. For us however there is no difference between the two kinds of octagons because we are really interested in the full symmetry group of the pattern generated by $\Gamma$, which includes the rotations $r_{j\pi/4}$, $j=1,\cdots ,8$, and therefore the boosts $r_{\pi/2} g_0^{-1}$, $g_1 r_{-\pi/2}$ and their conjugates by the rotation $r_\pi$, which are precisely the generators of the Gutzwiller lattice group.

We now determine what is the full symmetry group $G^*$ of the octagonal lattice, or equivalently, of the surface $\D/\Gamma$. Clearly the symmetry group of the octagon itself is part of it. This is the dihedral group $D_8$ generated by the rotation $r_{\pi/4}$ and by the reflection $\kappa$ through the real axis, but there is more. We have seen in Section \ref{section:lattices} that the group $G^*=\Lambda/\Gamma$, $\Lambda$ being the triangle group generated by reflections through the edges of a triangle $\tau$ which tiles (by the action of $\Lambda/\Gamma$) the surface $\D/\Gamma$. The smallest triangle (up to symmetry) with these properties is the one shown in Figure \ref{fig:triangle}. It has angles $\pi/8$, $\pi/2$ and $\pi/3$ at vertices $P=O$ (the center of $\D$), $Q$, $R$ respectively, and its area is, by Gauss-Bonnet formula, equal to $\pi/24$. There are exactly 96 copies of $\tau$ filling the octagon, hence $|G^*|=96$. The index two subgroup $G$ of orientation-preserving transformations in $G^*$ has therefore 48 elements. In \cite{broughton:91} it has been found that $G\simeq GL(2,3)$, the group of invertible $2\times 2$ matrices over the 3 elements field $\Z_3$. In summary:

\begin{proposition}
The full symmetry group $G^*$ of $\D/\Gamma$ is $G\cup\kappa G$ where $G\simeq GL(2,3)$ has 48 elements.
\end{proposition}

The isomorphism between $GL(2,3)$ and $G$ can be built as follows. We use the notation $\Z_3=\{0,1,2\}$ and we call $\rho$ the rotation by $\pi/4$ centered at $P$ (mod $\Gamma$), $\sigma$ the rotation by $\pi$ centered at $Q$ (mod $\Gamma$) and $\epsilon$ the rotation by $2\pi/3$ centered at $R$ (mod $\Gamma$).  In the notations of Section \ref{section:lattices}, $a=\sigma$, $b=\epsilon$, $c=\rho$, and $\rho\sigma\epsilon = 1$. Then we can take
\begin{equation*}
\rho = \left(\begin{array}{cc}0&2\\2&2\end{array}\right),~\sigma = \left(\begin{array}{cc}2&0\\0&1\end{array}\right),~\epsilon=\left(\begin{array}{cc}2&1\\2&0\end{array}\right)
\end{equation*}
since these matrices satisfy the conditions $\rho^8=\sigma^2=\epsilon^3=Id$ and $\rho\sigma\epsilon = Id$. Note that $\rho^4=-Id$ where $Id$ is the identity matrix. We shall subsequently use this notation. 
The group $GL(2,3)$, therefore the group $G$, is made of 8 conjugacy classes which we list in Table \ref{table:GL(2,3)}, indicating one representative, the number of elements in each class and their order. This result is classical and can be found, e.g., in \cite{lang:93}.

\begin{table}[ht]   
\begin{center}
\begin{tabular}{|c|c|c|c|c|c|c|c|c|} \hline representative & $Id$ & $\rho$ & $\rho^2$ & $-Id$ & $\rho^5$ & $\sigma$ & $\epsilon$ & $-\epsilon$ \\\hline order & 1 & 8 & 4 & 2 & 8 & 2 & 3 & 6 \\\hline \# elements & 1 & 6 & 6 & 1 & 6 & 12 & 8 & 8 \\\hline \end{tabular}
\end{center}
\caption{Conjugacy classes of $G\simeq GL(2,3)$}\label{table:GL(2,3)}
\end{table}

We now turn to the full symmetry group $G^*$ which is generated by $G$ and $\kappa$, the reflection through the real axis in $\D$ and which maps the octagon $\Oct$ to itself. We write $\kappa'=\rho\kappa$ the reflection through the side $PR$ of the triangle $\tau$. Note that (i) $\kappa'$ preserves also $\Oct$, (ii) $\kappa''=\epsilon\kappa'=\sigma\kappa$ is the reflection through the third side $QR$. 

In what follows we rely on the group algebra software GAP \cite{Gap}. For this we have first identified a presentation for $G\simeq GL(2,3)$ considered as an abstract group, then a presentation for $G^*$. The presentation for $GL(2,3)$ can be obtained with the command "P := PresentationViaCosetTable(GL(2,3))" and the relations are shown with the command "TzPrintRelators(P)":

\begin{lemma}
(i) As an abstract group, $G$ is presented with two generators $a$ and $b$ and three relations $a^2=1$, $b^3=1$ and $(abab^{-1}ab^{-1})^{2}=1$. 
(ii) As an abstract group, $G^*$ is presented with three generators $a$, $b$ and $c$ and six relations: the three relations for $G$ plus the three relations $c^2=1$, $(ca)^2=1$ and $(cb)^2=1$. \\
(iii) These abstract elements can be identified with automorphisms of $D/\Gamma$ as follows: $a=\sigma$, $b=\epsilon$ and $c=\kappa''$.
\end{lemma} 

Applying the above lemma we find with GAP that the 96 elements group $G^*$ has 13 conjugacy classes which are listed in table \ref{table:G*_isomdirectes} for direct isometries and in table \ref{table:G*_inversions} for isometries which reverse orientation. GAP gives representatives of the conjugacy classes in the abstract presentation, which in general have complicated expressions. In some cases we have chosen other representatives, using in particular the 8-fold generator. To simplify some expressions in the tables \ref{table:G*_isomdirectes} to \ref{table:G*_sousgroupes} we also use the notations
$$\widehat{\sigma}=\epsilon\sigma\epsilon^{-1},~~\widetilde{\sigma}=\rho^2\sigma\rho^{-2},
$$
where $\widehat{\sigma}$ is the rotation by $\pi$ centered at $\widehat{S}$ (mod $\Gamma$) and $\widetilde{\sigma}$ is the rotation by $\pi$ centered at $\widetilde{S}$ (mod $\Gamma$), see figure \ref{fig:tesselation}.
\begin{table}   
\begin{center}
\begin{tabular}{|c|c|c|c|c|c|c|c|} \hline class number & 1 & 2 & 3 & 4 & 5 & 6 & 7 \\\hline representative & $Id$ & $\rho$ & $\rho^2$ & $-Id$ & $\sigma$ & $\epsilon$ & $-\epsilon$ \\\hline order & 1 & 8 & 4 & 2 & 2 & 3 & 6 \\\hline \# elements & 1 & 12 & 6 & 1 & 12 & 8 & 8 \\\hline \end{tabular}
\end{center}
\caption{Conjugacy classes of $G^*$, orientation preserving transformations}\label{table:G*_isomdirectes}
\end{table}

\begin{table}   
\begin{center}
\begin{tabular}{|c|c|c|c|c|c|c|c|} \hline class number & 8 & 9 & 10 & 11 & 12 & 13 \\\hline representative & $\kappa$ & $\kappa'$ & $\widehat{\sigma}\kappa$ & $\rho\widehat{\sigma}\kappa$  & $\epsilon\kappa$ & $-\epsilon\kappa$  \\\hline order & 2 & 2 & 8 & 4 & 12 & 12 \\\hline \# elements & 8 & 8 & 12 & 12 & 4 & 4 \\\hline \end{tabular}
\end{center}
\caption{Conjugacy classes of $G^*$, orientation reversing transformations}\label{table:G*_inversions}
\end{table}

We shall also need in Section \ref{section:octagonalplanforms} the list of subgroups of $G^*$ together with their decomposition in conjugacy classes (in $G^*$) in order to apply the trace formula (\ref{eq:formuletrace}). Here again we rely on GAP to obtain the necessary informations. 


Then representatives of each class are determined by inspection. These datas are listed in tables \ref{table:G_sousgroupes}  (subgroups of $G$) and \ref{table:G*_sousgroupes} (subgroups containing orientation reversing elements). The subgroups are listed up to conjugacy in $G^*$, the subgroups of order two are not listed. The rationale for the notations is as follows: 
\begin{itemize}
\item $G_0$ is an index 2 subgroup of $G$. Seen as a subgroup of $GL(2,3)$ it is $SL(2,3)$, the subgroup of determinant 1 matrices. It contains no order 2 elements except $-Id$ and no order 8 elements.
\item $C_n$, $\widetilde{C}_n$, $C'_n$, denote order $n$ cyclic groups. The notation $C_n$ is standard for the $n$-fold rotation group centered at the origin.
\item $D_n$ denotes a group isomorphic to the dihedral group of order $2n$, generated by an $n$-fold rotation and a reflection. Hence $D_8$ is the symmetry group of the octagon. The notation $\widetilde{D}_n$ is used for a $2n$ element group which has an $n$-fold rotation and a $2$-fold rotation as generators. For example $\widetilde{D}_8=<\rho,~\widehat{\sigma}>$, and one can verify that $\widehat{\sigma}\rho\widehat{\sigma}^{-1}=\rho^3$, which makes $\widetilde{D}_8$ a {\em quasidihedral group} (see \cite{gorenstein:80}). 
\item $Q_8$ is a usual notation for the 8 elements quaternionic group.
\item The notation $H_{n\kappa}$ indicates a group generated by the group $H_n$ and $\kappa$. Same thing if replacing $\kappa$ by $\kappa'$. For example $\widetilde{C}_{3\kappa'}$ is the 6 elements group generated by $\widetilde{C}_{3}$ and $\kappa'$.
\end{itemize}

\begin{table}  
\begin{center}
\begin{tabular}{|c|c|c|c|} \hline Subgroup & Order & Generators & Subclasses: representatives (\# elements) \\\hline  $G_0\simeq SL(2,3)$ & 24 & $<\rho^2,~\epsilon>$ & \{$Id$ (1), $-Id$ (1), $\rho^2$ (6), $\epsilon$ (8), $-\epsilon$ (8)\} \\\hline $\widetilde{D}_8$ & 16 & $<\rho,~\widehat{\sigma}>$ & \{$Id$ (1), $-Id$ (1), $\rho$ (4), $\rho^2$ (6), $\widehat{\sigma}$ (4)\} \\\hline $\widetilde{D}_6$ & 12 & $<-\epsilon,~\widetilde{\sigma}>$ & \{$Id$ (1), $-Id$ (1), $\widetilde{\sigma}$ (6), $\epsilon$ (2), $-\epsilon$ (2)\} \\\hline $C_8$ & 8 & $<\rho>$ & \{$Id$ (1), $-Id$ (1), $\rho$ (4), $\rho^2$ (2)\} \\\hline $Q_8$ & 8 & $<\rho^2,~\sigma\rho^2\sigma>$ & \{$Id$ (1), $-Id$ (1), $\rho^2$ (6)\}  \\\hline $\widetilde{D}_4$ & 8 & $<\rho^2,~\widehat{\sigma}>$ & \{$Id$ (1), $-Id$ (1), $\rho^2$ (2), $\widehat{\sigma}$ (4)\} \\\hline $\widetilde{C}_6$ & 6 & $<-\epsilon>$ & \{$Id$ (1), $-Id$ (1), $\epsilon$ (2), $-\epsilon$ (2)\} \\\hline $\widetilde{D}_3$ & 6 & $<\epsilon,~\widetilde{\sigma}>$ & \{$Id$ (1), $\epsilon$ (2), $\widetilde{\sigma}$ (3)\} \\\hline $C_4$ & 4 & $<\rho^2>$ & \{$Id$ (1), $-Id$ (1), $\rho^2$ (2)\} \\\hline $\widetilde{D}_2$ & 4 & $<-Id,\sigma>$ & \{$Id$ (1), $-Id$ (1), $\sigma$ (2)\} \\\hline $\widetilde{C}_3$ & 3 & $<\epsilon>$ & \{$Id$ (1), $\epsilon$ (2)\} \\\hline $C_2$ & 2 & $<-Id>$ & \{$Id$ (1), $-Id$ (1)\} \\\hline $\widetilde{C}_2$ & 2 & $<\sigma>$ & \{$Id$ (1), $\sigma$ (1)\} \\\hline \end{tabular}
\end{center}
\caption{Subgroups of $G\subset G^*$  (up to conjugacy). The last column provides datas about their conjugacy subclasses (in $G^*$).}\label{table:G_sousgroupes} 
\end{table}

\begin{table}   
\begin{center}
\begin{tabular}{|c|c|c|c|} \hline Subgroup & Order & Generators & Subclasses: representatives (\# elements) \\\hline  $G_{0\kappa}$ & 48 & $<G_0,~\kappa>$ & $G_0$ $\cup$ \{$\kappa$ (6), $\rho\widehat{\sigma}\kappa$ (2), $\epsilon\kappa$ (8),  $-\epsilon\kappa$ (8)\}   \\\hline  $G_{0\kappa'}$ & 48 & $<G_0,~\kappa'>$ & $G_0$ $\cup$ \{$\kappa'$ (12), $\widehat{\sigma}\kappa$ (12)\}  \\\hline  $\widetilde{D}_{8\kappa}$ & 32 & $<\widetilde{D}_8,~\kappa>$ & $\widetilde{D}_8$ $\cup$ \{$\kappa$ (6), $\kappa'$ (4), $\widehat{\sigma}\kappa$ (4), $\rho\widehat{\sigma}\kappa$ (2) \} \\\hline  $\widetilde{D}_{6\kappa'}$ & 24 & $<\widetilde{D}_6,~\kappa'>$ & $\widetilde{D}_6$ $\cup$ \{$\kappa'$ (6), $\epsilon\kappa$ (2), $-\epsilon\kappa$ (2), $\rho\widehat{\sigma}\kappa$ (2)\} \\\hline $C_{8\kappa}$ (=${D}_{8})$ & 16 & $<C_8,~\kappa>$ & $C_8$ $\cup$ \{$\kappa$ (4), $\kappa'$ (4)\} \\\hline $C'_{8\kappa}$ & 16 & $<\rho^2\sigma,~\kappa>$ & $C_8$\ $\cup$ \{$\kappa$ (2), $\rho\widehat{\sigma}\kappa$ (2), $\widehat{\sigma}\kappa$ (4)\} \\\hline ${Q}_{8\kappa}$ & 16 & $<{Q}_{8},~\kappa>$ & $Q_8$ $\cup$ \{$\kappa$ (6), $\rho\widehat{\sigma}\kappa$ (2)\} \\\hline ${Q}_{8\kappa'}$ & 16 & $<{Q}_{8},~\kappa'>$ & $Q_8$ $\cup$ \{$\kappa'$ (4), $\widehat{\sigma}\kappa$ (4)\} \\\hline $\widetilde{D}_{4\kappa}$ & 16 & $<\widetilde{D}_4,~\kappa>$ & $\widetilde{D}_4$ $\cup$ \{$\kappa$ (4), $\widehat{\sigma}\kappa$ (4)\} \\\hline $\widetilde{D}_{4\kappa'}$ & 16 & $<\widetilde{D}_4,~\kappa'>$ & $\widetilde{D}_4$ $\cup$  \{$\kappa$ (2), $\kappa'$ (4), $\rho\widehat{\sigma}\kappa$ (2)\}  \\\hline $C'_{12}$ & 12 & $<\epsilon\kappa>$ & $\widetilde{C}_6$ $\cup$  \{$\epsilon\kappa$ (2), $-\epsilon\kappa$ (2), $\rho\widehat{\sigma}\kappa$ (2)\}  \\\hline $\widetilde{C}_{6\kappa'}$ & 12 & $<\widetilde{C}_6, \kappa'>$ & $\widetilde{C}_6$ $\cup$  \{$\kappa'$ (6)\} \\\hline $C'_8$ &8 & $<\widehat{\sigma}\kappa>$ & $C_4$ $\cup$  \{$\widehat{\sigma}\kappa$ (4)\} \\\hline $C_{4\kappa}$ (= $D_4$) & 8 & $<C_4, \kappa>$ & $C_4$ $\cup$  \{$\kappa$ (4)\} \\\hline $C_{4\kappa'}$ & 8 & $<C_4, \kappa'>$ & $C_4$ $\cup$  \{$\kappa'$ (4)\} \\\hline $\widetilde{D}_{2\kappa}$ & 8 & $<\widetilde{D}_2, \kappa>$ & $\widetilde{D}_2$ $\cup$  \{$\kappa$ (2), $\kappa'$ (2)\} \\\hline $C'_{4\kappa}$ & 8 & $<C'_4, \kappa>$ & $C'_4$ $\cup$  \{$\rho^2$ (2), $\kappa$ (2)\} \\\hline $C'_{4\kappa'}$ & 8 & $<C'_4, \kappa'>$ & $C'_4$ $\cup$  \{$\sigma$ (2), $\kappa'$ (2)\} \\\hline $\widetilde{C}_{3\kappa'}$ & 6 & $<\widetilde{C}_3, \kappa'>$ & $\widetilde{C}_3$ $\cup$  \{$\kappa'$ (3)\} \\\hline $C'_4$ & 4 & $<\rho\widehat{\sigma}\kappa>$ & \{$Id$ (1), $-Id$ (1), $\rho\widehat{\sigma}\kappa$ (2)\} \\\hline $C_{2\kappa}$ & 4 & $<-Id,\kappa>$ & \{$Id$ (1), $-Id$ (1), $\kappa$ (2)\} \\\hline $C_{2k'}$ & 4 & $<-Id,\kappa'>$ & \{$Id$ (1), $-Id$ (1), $\kappa'$ (2)\} \\\hline $\widetilde{C}_{2k}$ & 4 & $<\sigma,\kappa>$ & \{$Id$ (1), $\sigma$ (1), $\kappa$ (1), $\kappa'$ (1)\} \\\hline $\widetilde{C}'_{2\kappa}$ & 4 & $<\widetilde{\sigma},\kappa>$ & \{$Id$ (1), $\widetilde{\sigma}$ (1), $\kappa$ (1), $\kappa'$ (1)\} \\\hline $C_{1\kappa}$ & 2 & $<\kappa>$ & \{$Id$ (1), $\kappa$ (1)\} \\\hline $C_{1\kappa'}$ & 2 & $<\kappa'>$ & \{$Id$ (1), $\kappa'$ (1)\} \\\hline \end{tabular}
\end{center}
\caption{Subgroups of $G^*$, not in $G$ (up to conjugacy). The last column provides datas about their conjugacy subclasses (in $G^*$). $\widetilde{C}_4$ is a subgroup conjugate to $C_4$ with generator $(\rho^2\sigma)^2$.}\label{table:G*_sousgroupes}
\end{table}

\subsection{The irreducible representations of $G^*$}
There are 13 conjugacy classes and therefore we know there are 13 complex irreducible representations of $G^*$, the characters of which will be denoted $\chi_j$, $j=1,...,13$. The character table, as computed by GAP, is shown in table \ref{table:caracteres}.

\begin{table}  
\begin{center}
\begin{tabular}{|c|c|c|c|c|c|c|c|c|c|c|c|c|c|}\hline Class \# & 1 & 2 & 3 & 4 & 5 & 6 & 7 & 8 & 9 & 10 & 11 & 12 & 13  \\\hline Representative & $Id$ & $\rho$ & $\rho^2$ & $-Id$ & $\sigma$ & $\epsilon$ & $-\epsilon$ & $\kappa$ & $\kappa'$ & $\widehat{\sigma}\kappa$ & $\rho\widehat{\sigma}\kappa$  & $\epsilon\kappa$ & $-\epsilon\kappa$  \\\hline\hline $\chi_1$ &  1 & 1 & 1 & 1 & 1 & 1 & 1 & 1 & 1 & 1 & 1 & 1 & 1 \\\hline $\chi_2$ & 1 & -1 & 1 & 1 & -1 & 1 & 1 & 1 & -1 & -1 & 1 & 1 & 1 \\\hline $\chi_3$ & 1 & -1 & 1 & 1 & -1 & 1 & 1 & -1 & 1 & 1 & -1 & -1 & -1 \\\hline $\chi_4$ & 1 & 1 & 1 & 1 & 1 & 1 & 1 & -1 & -1 & -1 & -1 & -1 & -1 \\\hline $\chi_5$ & 2 & 0 & 2 & 2 & 0 & -1 & -1 & -2 & 0 & 0 & -2 & 1 & 1 \\\hline $\chi_6$ & 2 & 0 & 2 & 2 & 0 & -1 & -1 & 2 & 0 & 0 & 2 & -1 & -1 \\\hline $\chi_7$ & 3 & 1 & -1 & 3 & -1 & 0 & 0 & -1 & -1 & 1 & 3 & 0 & 0 \\\hline $\chi_8$ & 3 & 1 & -1 & 3 & -1 & 0 & 0 & 1 & 1 & -1 & -3 & 0 & 0 \\\hline $\chi_9$ & 3 & -1 & -1 & 3 & 1 & 0 & 0 & 1 & -1 & 1 & -3 & 0 & 0 \\\hline $\chi_{10}$ & 3 & -1 & -1 & 3 & 1 & 0 & 0 & -1 & 1 & -1 & 3 & 0 & 0 \\\hline $\chi_{11}$ & 4 & 0 & 0 & -4 & 0 & -2 & 2 & 0 & 0 & 0 & 0 & 0 & 0 \\\hline $\chi_{12}$ & 4 & 0 & 0 & -4 & 0 & 1 & -1 & 0 & 0 & 0 & 0 & $\sqrt{3}$ & $-\sqrt{3}$ \\\hline $\chi_{13}$ & 4 & 0 & 0 & -4 & 0 & 1 & -1 & 0 & 0 & 0 & 0 & $-\sqrt{3}$ & $\sqrt{3}$ \\\hline \end{tabular}
\end{center}
\caption{Irreducible characters of $G^*$}\label{table:caracteres} 
\end{table}

The character of the identity is equal to the dimension of the corresponding representation. It follows from table \ref{table:caracteres} that there are 4 irreducible representations of dimension 1, 2 of dimension 2, 4 of dimension 3 and 3 of dimension 4. In the following we shall denote the irreducible representations by their character: $\chi_j$ is the representation with this character.

\begin{lemma}
All irreduclible representations of $G^*$ listed in table \ref{table:caracteres} are real absolutely irreducible.
\end{lemma}
\begin{proof}
This is clear for the one dimensional representations whose characters are real. \\
For the two dimensional representations, let us consider the dihedral subgroup $D_3$ generated by the 3-fold symmetry $\epsilon$ and the reflection $\kappa'$. The representation of $D_3$ in either representations planes of $\chi_5$ and $\chi_6$ have characters $\chi_j(\epsilon)=-1$ and $\chi_j(\kappa')=0$ ($j=5$ or $6$). These are the characters of the 2D irreducible representation of $D_3$, which is absolutely irreducible, and therefore the representations $\chi_5$ and $\chi_6$ of $G^*$ are also absolutely irreducible. Indeed if any real linear map which commutes with the elements of a subgroup is a scalar multiple of the identity, this is a fortiori true for the maps which commute with the full group. \\
For the three dimensional representations $\chi_7$ to $\chi_{10}$, let us first remark that if we write $C_2=\{Id,-Id\}$, then $G/C_2\simeq\mathbb{O}$, the octahedral group. Its subgroup $\mathbb{T}$ (tetrahedral group) can easily be identified with the 12 elements group generated by the "pairs" $\{Id,-Id\}$, $\{\epsilon,-\epsilon\}$ and $\{\rho^2, -\rho^2\}$. Now we consider the representation of $G$ defined by the action of $\chi_j$ restricted to $G$ (for each 3D $\chi_j$).  One can check easily from the character table that it projects onto a representation of $G^*/C_2$, the character of which is given by the value of $\chi_j$ on the corresponding conjugacy classes, and in particular the character for the representation of the group $\mathbb{T}$ is given, for any $j=7$ to $10$, by $\chi_j(\{Id,-Id\})=3$, $\chi_j(\{\epsilon,-\epsilon\})=0$ and $\chi_j(\{\rho^2, -\rho^2\})=-1$. But this is the character of the irreducible representation of $\mathbb{T}$  \cite{miller:72}, which is absolutely irreducible (natural action of  $\mathbb{T}$ in $\R^3$). Hence the three dimensional representations of $G^*$ are absolutely irreducible by the same argument as above.\\
It remains to prove the result for the four dimensional representations $\chi_{11}$, $\chi_{12}$ and $\chi_{13}$. For this we consider the action of the group $D_8$ generated by $\rho$ and $\kappa$, as defined by either one of these 4D irreducible representations of $G^*$. We observe from the character table that in all cases, the character of this action is $\chi(\rho)=0$, $\chi(\rho^2)=0$, $\chi(-Id)=-4$, $\chi(\rho^3)=0$ ($\rho$ and $\rho^3$ are conjugate in $G^*$), and $\chi(\kappa)=\chi(\kappa')=0$. We can determine the isotypic decomposition for this action of $D_8$ from these character values. The character tables of the four one dimensional and three two dimensional irreducible representations of $D_8$ can be computed easily either by hand (see \cite{miller:72} for the method) or using a computer group algebra software like GAP. For all one dimensional characters the value at $-Id$ is $1$, while for all two dimensional characters, the value at $-Id$ is $-2$. Since $\chi(-Id)=-4$, it is therefore not possible to have one dimensional representations in this isotypic decomposition. It must therefore be the sum of two representations of dimension 2. Moreover, since $\chi(\rho)=\chi(\rho^2)=\chi(\rho^3)=0$, it can't be twice the same representation. In fact it must be the sum of the representations whose character values at $\rho$ are $\sqrt{2}$ and $-\sqrt{2}$ respectively. Now, these representations are absolutely irreducible (well-know fact which is straightforward to check), hence any $D_8$-equivariant matrix which commutes with this action decomposes into a direct sum of two scalar $2\times 2$ matrices $\lambda I_2$ and $\mu I_2$ where $\lambda$ and $\mu$ are real. But the representation of $G^*$ is irreducible, hence $\lambda=\mu$, which proves that it is also absolutely irreducible. 
\end{proof}

\subsection{The octagonal H-planforms \label{section:octagonalplanforms}}
We can now apply Lemma \ref{lemme:equivariantbif} in order to determine the H-planforms for the octagonal lattice.

\begin{theorem}\label{theorem:repirr}
The irreducible representations of $G^*$ admit H-planforms with the following isotropy types:
\begin{itemize}
\item $\chi_1$: $G^*$;
\item $\chi_2$: $G_{0\kappa}$;
\item $\chi_3$: $G_{0\kappa'}$;
\item $\chi_4$: $G\simeq GL(2,3)$;  
\item $\chi_5$: $\widetilde{D}_8$, $Q_{8\kappa'}$;
\item $\chi_6$: $\widetilde{D}_{8\kappa}$;
\item $\chi_7$: $C'_{8\kappa}$, $C'_{12}$, $C_{4\kappa'}$;
\item  $\chi_8$: $C_{8\kappa}$, $\widetilde{C}_{6\kappa'}$, $\widetilde{D}_{2\kappa}$;
\item $\chi_9$: $\widetilde{D}_6$, $\widetilde{D}_{4\kappa}$;
\item $\chi_{10}$: $\widetilde{D}_{6\kappa'}$, $\widetilde{D}_{4\kappa'}$;
\item $\chi_{11}$: $\widetilde{C}_{2\kappa}$, $\widetilde{C}'_{2\kappa}$;
\item $\chi_{12}$: $\widetilde{D}_3$, $\widetilde{C}_{3\kappa'}$, $\widetilde{C}_{2\kappa}$, $\widetilde{C}'_{2\kappa}$;
\item $\chi_{13}$: $\widetilde{D}_3$, $\widetilde{C}_{3\kappa'}$, $\widetilde{C}_{2\kappa}$, $\widetilde{C}'_{2\kappa}$;
\end{itemize}

\end{theorem}
\begin{proof}
For the one dimensional representations of $G^*$ this is straightforward: each element whose character image is $+1$ belongs to the isotropy group. The result follows therefore directly from the character table and list of subgroups of $G^*$. For the higher dimensional irreducible representations we need to find those isotropy subgroups $H$ such that (see \eqref{eq:formuletrace}):
$$
1 = {\rm dim}(V^H) = \frac{1}{|H|} \sum_{h\in H}\chi_j(h)
$$
where $\chi_j$ denotes the character of the $j$-th irreducible representation.
This can be done in a systematic way by using the character table \ref{table:caracteres} and applying the datas on subgroups and their conjugacy classes listed in \ref{table:G_sousgroupes} and \ref{table:G*_sousgroupes}. The calculations are cumbersome but can be slightly simplified by noting that if $H\subset H'$ and ${\rm dim}(V^H)=0$ (a case which occurs many times), then not only $H$ is not symmetry breaking but also $H'$, since $H\subset  H'\Rightarrow V^{H'}\subset V^H$. \\
The following lemma is also useful as it eliminates most candidates in the case of 4D representations:
\begin{lemma}
If a subgroup $H$ contains $-Id$, then for $j=11, 12$ or $13$, one has $\sum_{h\in H}\chi_j(h)=0$.
\end{lemma}  
\noindent {\em Proof of the lemma.} In all three cases the result follows from the relations: (i) $\chi_j(-Id)=-\chi_j(Id)$, (ii) $\chi(-\epsilon)=-\chi(\epsilon)$ and $\chi(-\epsilon\kappa)=-\chi(\epsilon\kappa)$, (iii) $\chi_j(s)=0$ for all $s$ which is not conjugate to one in (i) or (ii).  \\
\end{proof}

For the lower dimensional representations it is possible to reduce the problem to known situations and to provide bifurcation diagrams without any further calculations. The next theorem provides these informations. Stability of the solutions has to be understood here with respect to perturbations with the same octagonal periodicity in $\D$ and under the condition that, in $L^2(\D/\Gamma)$, the corresponding representation corresponds to the "most unstable" modes ("neutral modes" at bifurcation). 

\begin{theorem}\label{theorem:bifdiagram}
For the one and two dimensional representations, the generic bifurcation diagrams have the following properties:
\begin{itemize}
\item $\chi_1$: transcritical branch, exchange of stability principle holds.
\item $\chi_2$, $\chi_3$ and $\chi_4$: pictchfork bifurcation, exchange of stability principle holds.
\item $\chi_5$: same as bifurcation with hexagonal symmetry in the plane, see Figure \ref{fig: diagbif_D6}.
\item $\chi_6$: same as bifurcation with triangular symmetry in the plane. In particular H-planforms are always unstable on both sides of the bifurcation point unless the subcritical branch bends back sufficiently near the bifurcation point (see Figure \ref{fig: diagbif_D3}).
\end{itemize}
\end{theorem}
\begin{proof}
(i) For the one dimensional representations, this follows from classical bifurcation theory: in $\chi_1$ there is no symmetry breaking, hence generically the bifurcation is of transcritical type and the trivial and bifurcated solutions exchange stability at the bifurcation point. In the three other cases, a symmetry exists which acts by reversing direction on the axis as can be seen from Table \ref{table:caracteres}. For example in $\chi_2$ this can be taken as $\sigma$ (but also $\kappa'$ does the same thing). Hence the bifurcation is of pitchfork type and exchange of stability holds. \\
(ii) For $\chi_5$, note that the subgroup $Q_8$ acts trivially on any point of this plane (${\rm dim}(V^{Q_8})=2$). In fact $Q_8$ is the isotropy group of the principal stratum in this group action. Now, $G\simeq GL(2,3)=Q_8\ltimes D_3$, hence $G^*/Q_8\simeq D_3\ltimes \Z_2\simeq D_6$, the symmetry group of an hexagon. This group action is isomorphic to the natural action of $D_6$ in the plane. It follows that the problem reduces in this case to a bifurcation problem with the action of $D_6$ in the plane, see \cite{golubitsky-stewart-etal:88} for details. \\
(iii)  In the case of $\chi_6$, the maximal subgroup which keeps every point in the plane fixed is the 16 element group $Q_{8\kappa}$. It follows that the problem reduces to a bifurcation problem in the plane with symmetry $G^*/Q_{8\kappa} \simeq D_3$. Details on this bifurcation can be found in  \cite{golubitsky-stewart-etal:88}.
\end{proof}
\begin{figure}[htp]
\centering
\includegraphics[width=0.8\textwidth]{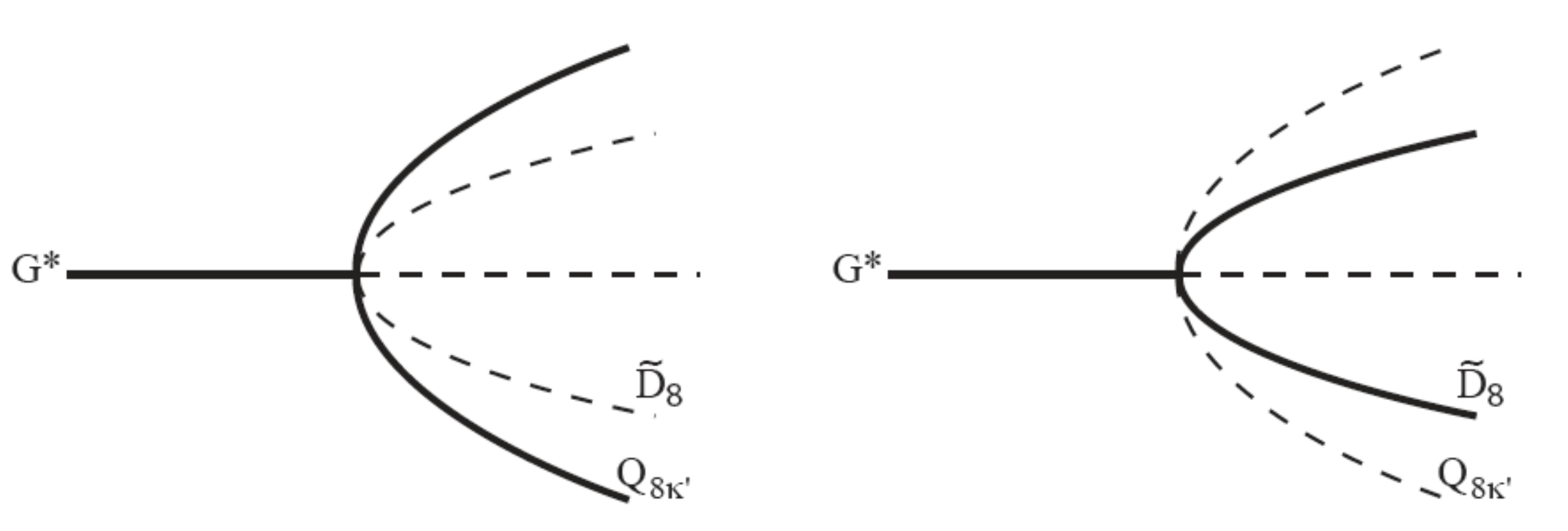}
\caption{Bifurcation diagram for the case $\chi_5$. Dotted lines: unstable branches, .}
\label{fig: diagbif_D6}
\end{figure}
\begin{figure}[htp]
\centering
\includegraphics[width=0.4\textwidth]{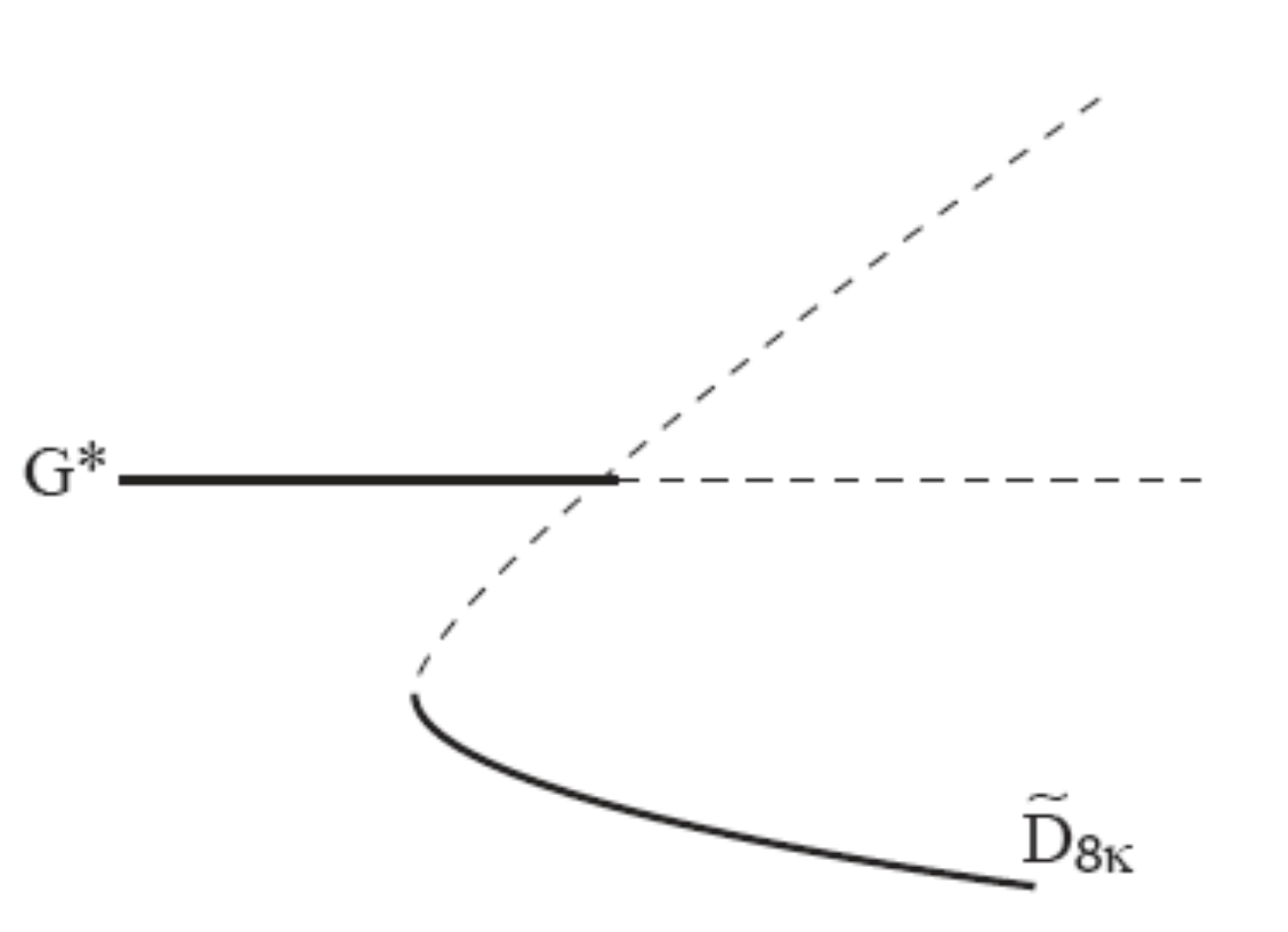}
\caption{Bifurcation diagram for the case $\chi_6$. Dotted lines: unstable branches.}
\label{fig: diagbif_D3}
\end{figure}

{\em Remark 1:} A similar reduction can be made with the 3 dimensional representations. Indeed it can be seen that the principal isotropy type (which keeps all points in the three dimensional representation space fixed) is $C'_4$ for $\chi_7$ and $\chi_{10}$, and $C_2$ for   $\chi_8$ and $\chi_9$. In the first case, this leads to reducing the problem to one with $G^*/C'_4\simeq \mathbb{O}$ symmetry, where       $\mathbb{O}$ is the 24 elements group of direct symmetries (rotations) of a cube. However there are two irreducible representations of dimension 3 of $\mathbb{O}$ \cite{miller:72}, and it turns out that $\chi_7$ corresponds to one of them (the "natural" action of $\mathbb{O}$ in $\R^3$) while $\chi_{10}$ corresponds to the other representation. This explains why there are 3 types of H-planforms for $\chi_7$ and only 2 for $\chi_{10}$. Similarly, the principal isotropy type for $\chi_8$ and $\chi_9$ is the two element group $C_2$, and $G^*/C_2\simeq \mathbb{O}\ltimes \Z_2$. Then the same remark holds for these cases as for the previous ones. 

{\em Remark 2:} In the 4 dimensional cases, the principal isotropy type is the trivial group, hence no reduction can be made. Bifurcation in this case (and in the 3 dimensional cases as well) will be the subject of a forthcoming paper.

\section{Computing the H-planforms} \label{section:computing}


It follows from the definition that H-planforms are eigenfunctions of the Laplace-Beltrami operator in $\D$ which satisfy certain isotropy conditions: (i) being invariant under a lattice group $\Gamma$ and (ii) being invariant under the action of an isotropy subgroup of the symmetry group of the fundamental domain $\D/\Gamma$ (mod $\Gamma$). Therefore in order to exhibit H-planforms, we need first to compute eigenvalues and eigenfunctions of $\triangle$ in $\D$, and second to find those eigenfunctions which satisfy the desired isotropy conditions. In this section we tackle this question in the case where the lattice has the regular octagon as a fundamental domain.

Over the past decades, computing the eigenmodes of the Laplace-Beltrami operator on compact manifolds has received much interest from physicists. The main applications are certainly in quantum chaos \cite{balazs-voros:86,aurich-steiner:89,aurich-steiner:93,schmit:91,cornish-turok:98} and in cosmology \cite{inoue:99,cornish-spergel:99,lehoucq-weeks-etal:02}.

To our knowledge, the interest in such computation was sparkled by the study of classical and quantum mechanics on surfaces of constant negative curvature, and the connections between them (for an overview on the subject see \cite{balazs-voros:86}). To be more precise, \textit{quantum chaology} can be defined as the study of the semiclassical behaviour characteristic of systems whose classical motion exhibits chaos, for example the classical free motion of a mass point on a compact surface of constant negative curvature (as it is the most chaotic possible). In \cite{balazs-voros:86,aurich-steiner:89,cornish-turok:98}, authors studied the time-independent Schroedinger equation on the compact Riemannian surface of constant curvature -1 and genus 2, which is topologically equivalent to the regular octagon with four periodic boundary conditions. This is the same as solving the eigenvalue problem for $\Gamma$ invariant eigenmodes in $\D$. The first computations have been performed using the finite element method on "desymmetrised" domains of the hyperbolic octagon with a mixture of Dirichlet and Neumann boundary conditions \cite{balazs-voros:86,schmit:91}. We explain the procedure of desymmetrisation in the next subsection. Aurich and Steiner in \cite{aurich-steiner:89} were the first to compute the eigenmodes on the whole octagon with periodic boundary conditions. They began with the finite element method of type $P2$ and were able to exhibit the first 100 eigenvalues. In \cite{aurich-steiner:93}, the same authors used the direct boundary-element method on an asymmetric octagon to reach the 20 000th eigenvalue. 

There is also a strong interest of cosmologists for ringing the eigenmodes of the Laplace-Beltrami operator on compact surfaces. Indeed this is necessary in order to evaluate the cosmic microwave background anisotropy in multiply-connected compact cosmological models. For some models, this computation is performed on a compact hyperbolic 3-space called the Thurston manifold, and Inoue computed the first eigenmodes of Thurston space such that each corresponding eigenvalue $\lambda$ satisfies $\lambda\leq 10$ 
 with the direct boundary-element method \cite{inoue:99}. For 3-dimensional spherical spaces, several methods have been proposed: the ``ghosts method'' \cite{cornish-spergel:99}, the averaging method and the projection method. All these methods are explained and summarized in \cite{lehoucq-weeks-etal:02}.

Our aim is different in that we do not want to compute all the eigenvalues of the Laplace-Beltrami operator, but instead to calculate the H-planforms with the isotropy types listed in theorem \ref{theorem:repirr}. The methods of numerical computation are however similar, and one question is to choose the method  best suited to our goal. For the H-planforms associated to irreductible representations of dimension 1 (i.e. for $(\chi_i)_{i=1\cdots 4}$), we use a desymmetrization of the octagon with a reformulation of the boundary conditions. For H-planforms associated with irreductible representations of dimension $\geq 2$, the desymmetrization of the octagon is also possible but much more complicated as noticed by Balazs and Voros \cite{balazs-voros:86}. This will therefore be the subject of another paper. Here we only identify some H-planforms of specific isotropy types. In order to find these H-planforms, we use the finite-element method with periodic boundary conditions. This choice is dictated by the fact that this method will allow us to compute all the first $n$ eigenmodes and among all these we will indentify those which correspond to a given isotropy group. As explained before, if we have used the direct boundary-element method, we would have reached any eigenmode but the search for H-planforms would also have became more random. Indeed, each iteration of this method gives only one eigenmode while the finite-element method provides $n$ eigenmodes depending on the precision of the discretization. This is why we prefer to use this last method in order to find some H-planforms associated with irreductible representations of dimension $\geq 2$, although it is quite more complicated to implement because of the periodic boundary conditions.

\subsection{Desymmetrization of the octagon}

We have already seen that the fundamental domain $T(2,3,8)$ of the group $G^*$ generates a tiling of the octagon. 
Desymmetrization consists in separating the individual solutions according to the symmetry classes of $G^*$. This entails solving the eigenvalue problem in certain irreducible subregions of the fundamental domain, such as $T(2,3,8)$, using special boundary conditions for these subregions. In effect the periodicity conditions in the original domain (octagon) may produce Dirichlet or Neumann conditions on the boundaries of these subregions. The symmetry group $G^*$ has a smaller fundamental domain, the triangle $T(2,3,8)$ see figure \ref{fig:triangle}, which is $\frac{1}{96}$th of the original octagon. The method of desymmetrization can be applied to many other techniques than finite element methods (see \cite{fassler-stiefel:92}).

We now focus on the four one-dimensional irreductible representations $(\chi_i)_{i=1\cdots 4}$ acting upon the generators as indicated in table \ref{table:caracteres}. One can find in the book of F{\"a}ssler and Stiefel \cite[Chapter 3]{fassler-stiefel:92} the principle of desymmetrization in the context of dihedral symmetry. We follow their method in the case of the symmetry group $G^*$. The first step is to attribute one number (value) to each of the 96 triangles that tesselate the octagon under the action of $G^*$ (see \ref{fig:tesselation}), according to the character values obtained from table \ref{table:caracteres}, i.e. $\pm 1$ depending on the conjugacy class (remember we restrict ourselves to the one dimensional representations $\chi_1$ to $\chi_4$). 

Let us take the example of the first irreductible representation $\chi_1$ and explain how we obtain the domain and the boundary conditions depicted in figure \ref{fig:tableau}. Table \ref{table:caracteres} shows that all 96 triangles end up with the same value, 1.  This means that the eigenfunction we are looking for is \textit{even} under all the 96 elements in $G^*$ and  it follows that it must satisfy Neumann boundary conditions on all the edges of the tesselation of the hyperbolic octagon \cite{balazs-voros:86,aurich-steiner:89}. Finally, it is sufficient to solve the eigenproblem on the reduced domain $T(2,3,8)$ with Neumann boundary conditions on its three edges. For the four one-dimensional representations one has to choose the correct combination of Neumann and Dirichlet boundary conditions as shown in figure \ref{fig:tableau}. 

The representations of dimension $\geq 2$ require the same number of values as their dimension. For example there are two basis vectors determining the function values in the case of an irreductible representation of dimension 2. The table \ref{table:caracteres} is then no longer sufficient to set the values of the function on each triangles and one has to explicitly write the matrices of the irreductible representation in order to obtain the suitable conditions. This is why we have restricted ourselves to the four one-dimensional representations as described in the previous paragraphs. 

\begin{figure}[htp]
\centering
\subfigure[$\chi_1:G^*$.]{
\label{fig:bc1}
\includegraphics[width=0.45\textwidth]{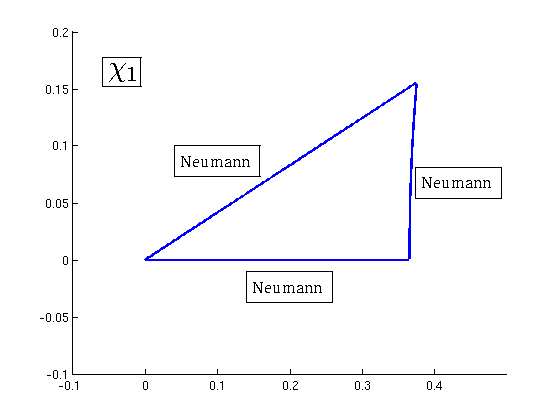}}
\hspace{.3in}
\subfigure[$\chi_2:G_{0\kappa}$.]{
\label{fig:bc2}
\includegraphics[width=0.45\textwidth]{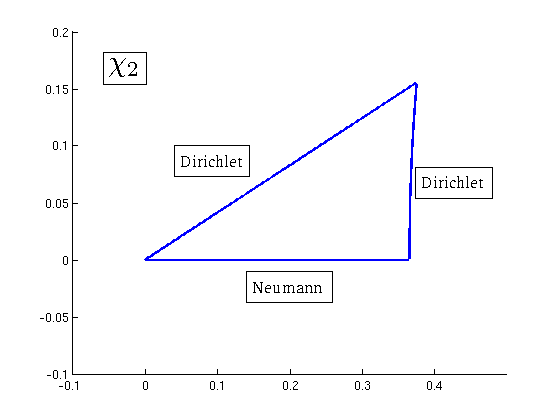}}\\
\subfigure[$\chi_3:G_{0\kappa'}$.]{
\label{fig:bc3}
\includegraphics[width=0.45\textwidth]{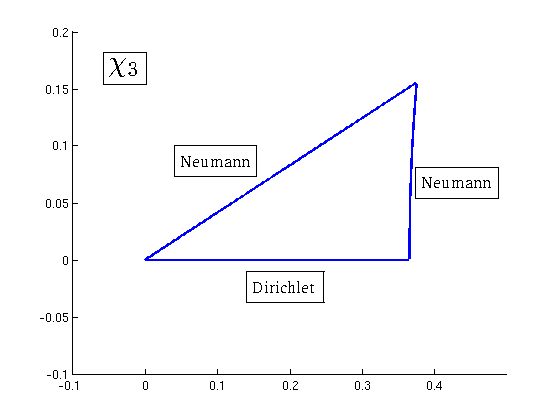}}
\hspace{.3in}
\subfigure[$\chi_4:G$.]{
\label{fig:bc4}
\includegraphics[width=0.45\textwidth]{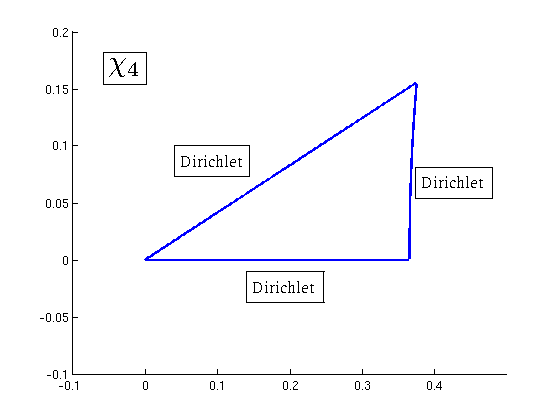}}
\caption{Boundary conditions for the one-dimensional irreducible representations. Top left: boundary conditions for $\chi_1$, corresponding to the isotropy group $G^*$. Top right: boundary conditions for $\chi_2$ corresponding to the isotropy group $G_{O\kappa}$. Bottom left: boundary conditions for $\chi_3$ corresponding to the isotropy group $G_{0\kappa'}$. Bottom right: boundary conditions for $\chi_4$ corresponding to the isotropy group $G$.}
\label{fig:tableau}
\end{figure}

\subsection{Numerical experiments}

As there exists an extensive literature on the finite element methods (see for an overview \cite{ciarlet-lions:91,allaire:05}) and as numerical analysis is not the main goal of this article, we do not detail the method itself but rather focus on the way to actually compute the eigenmodes of the Laplace-Beltrami operator.

\paragraph{Desymmetrized problem :} For the four problems depicted in figure \ref{fig:tableau}, we use the mesh generator \textit{Mesh2D} from \textit{Matlab} to tesselate the triangle $T(2,3,8)$ with 2995 nodes and we implement the finite element method of order 1. Our results are presented in figure \ref{fig:dim1} and are in a good agreement with those obtained by Balazs-Voros in \cite{balazs-voros:86} and Aurich-Steiner in \cite{aurich-steiner:89}. Once we have computed the eigenfunction in $T(2,3,8)$, we extend it to the whole octagon by applying the generators of $G^*$. We superimpose in figure \ref{fig:chi1} the tesselation of the octagon by the 96 triangles in order to allow the reader to see the symmetry class of $G^*$.

\begin{figure}[htbp]
\centering
\subfigure[$\chi_1:G^*$, the corrseponding eigenvalue is $\lambda=23.0790$.]{
\label{fig:chi1}
\includegraphics[width=0.4\textwidth]{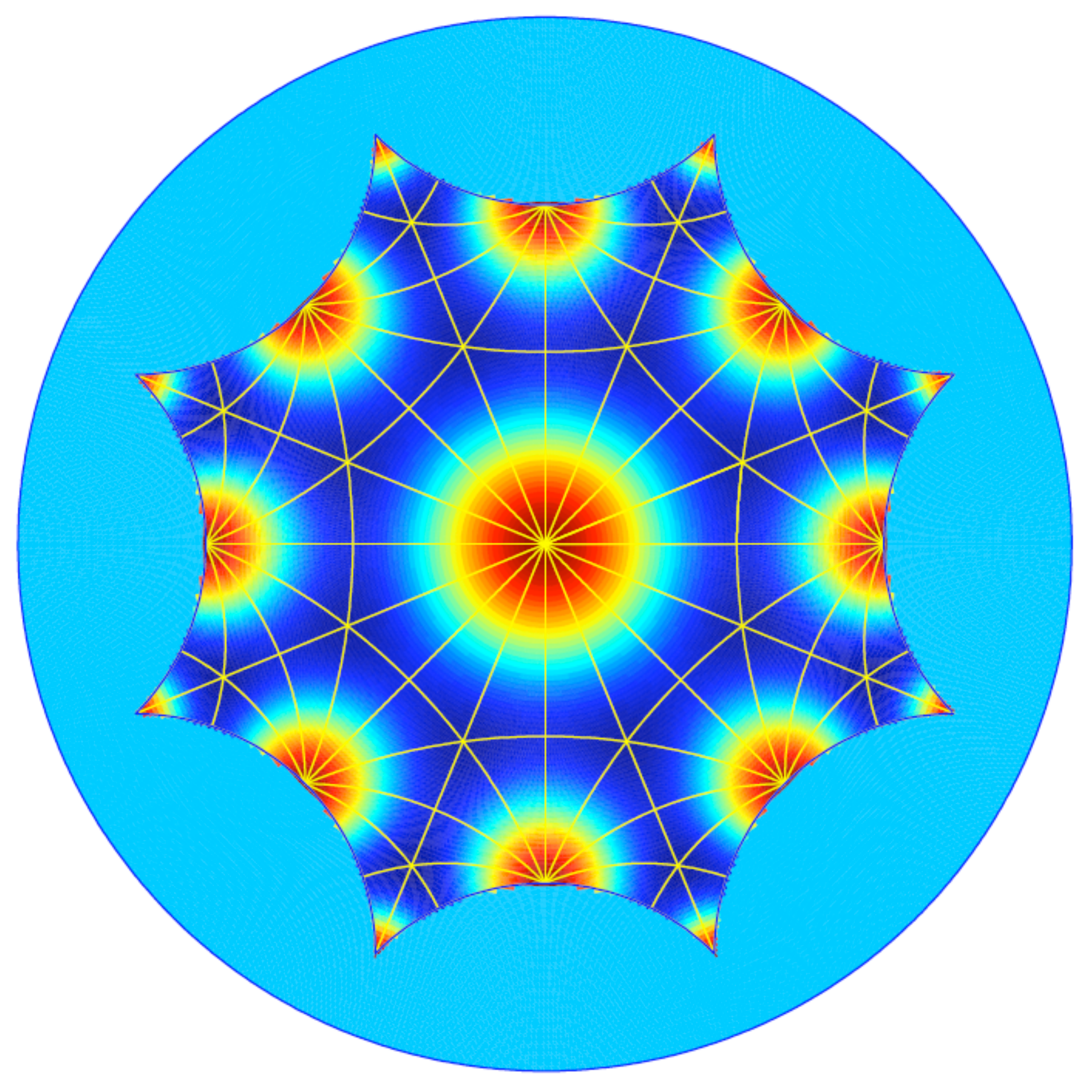}}
\hspace{.3in}
\subfigure[$\chi_2:G_{0\kappa}$, the corrseponding eigenvalue is $\lambda=91.4865$.]{
\label{fig:chi2}
\includegraphics[width=0.4\textwidth]{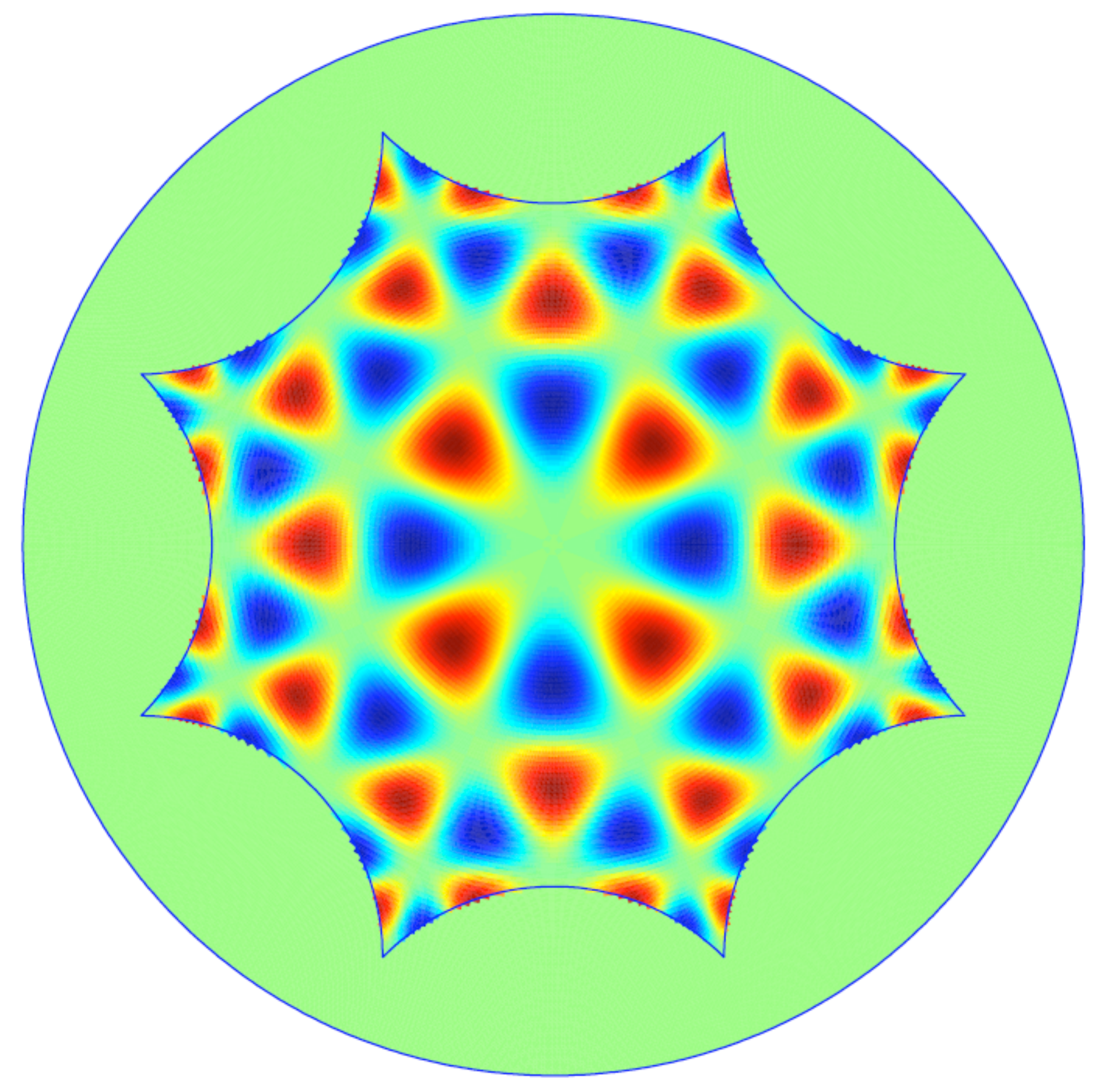}}\\
\subfigure[$\chi_3:G_{0\kappa'}$, the corrseponding eigenvalue is $\lambda=32.6757$.]{
\label{fig:chi3}
\includegraphics[width=0.4\textwidth]{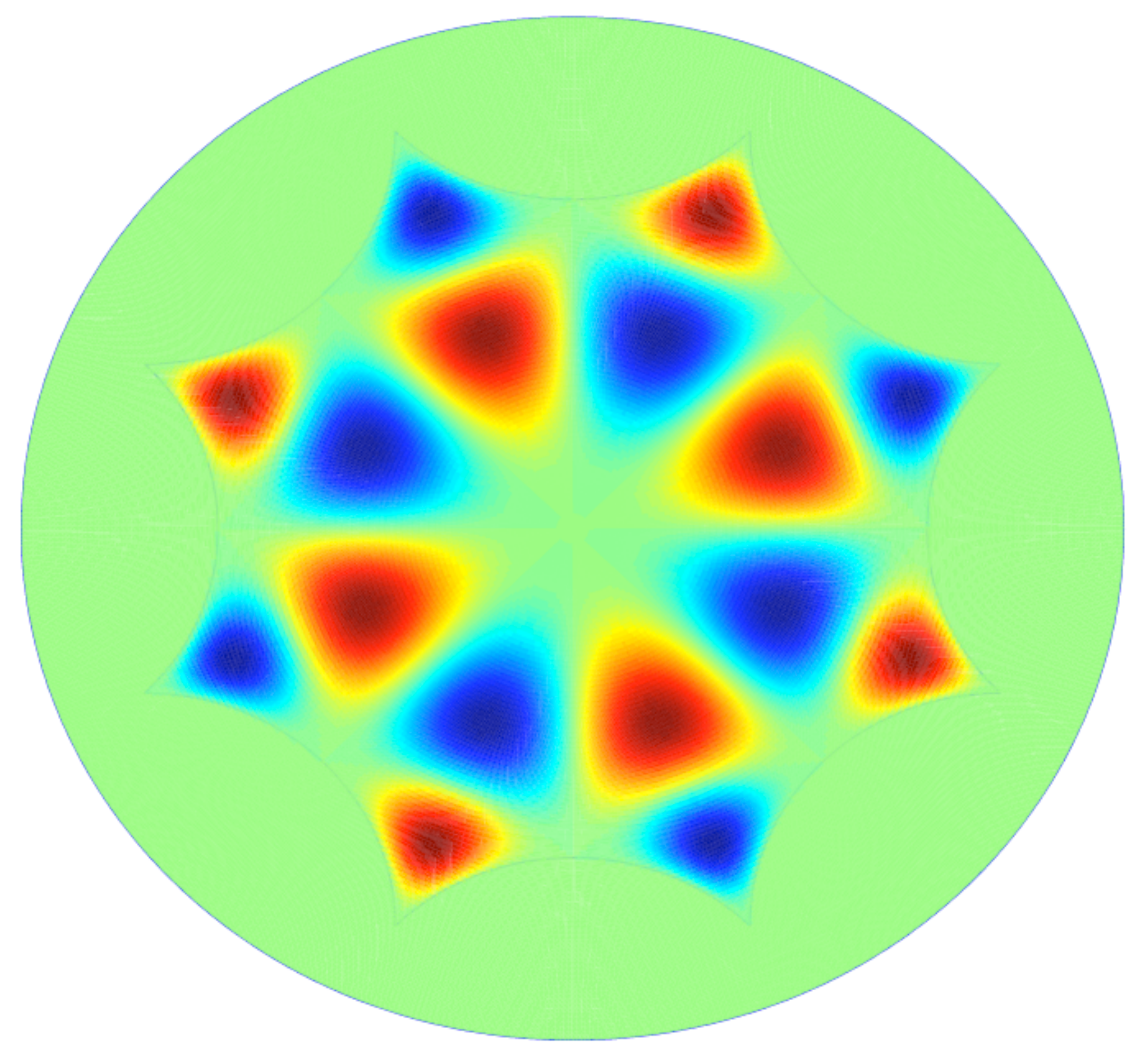}}
\hspace{.3in}
\subfigure[$\chi_4:G$, the corresponding eigenvalue is $\lambda=222.5434$.]{
\label{fig:chi4}
\includegraphics[width=0.4\textwidth]{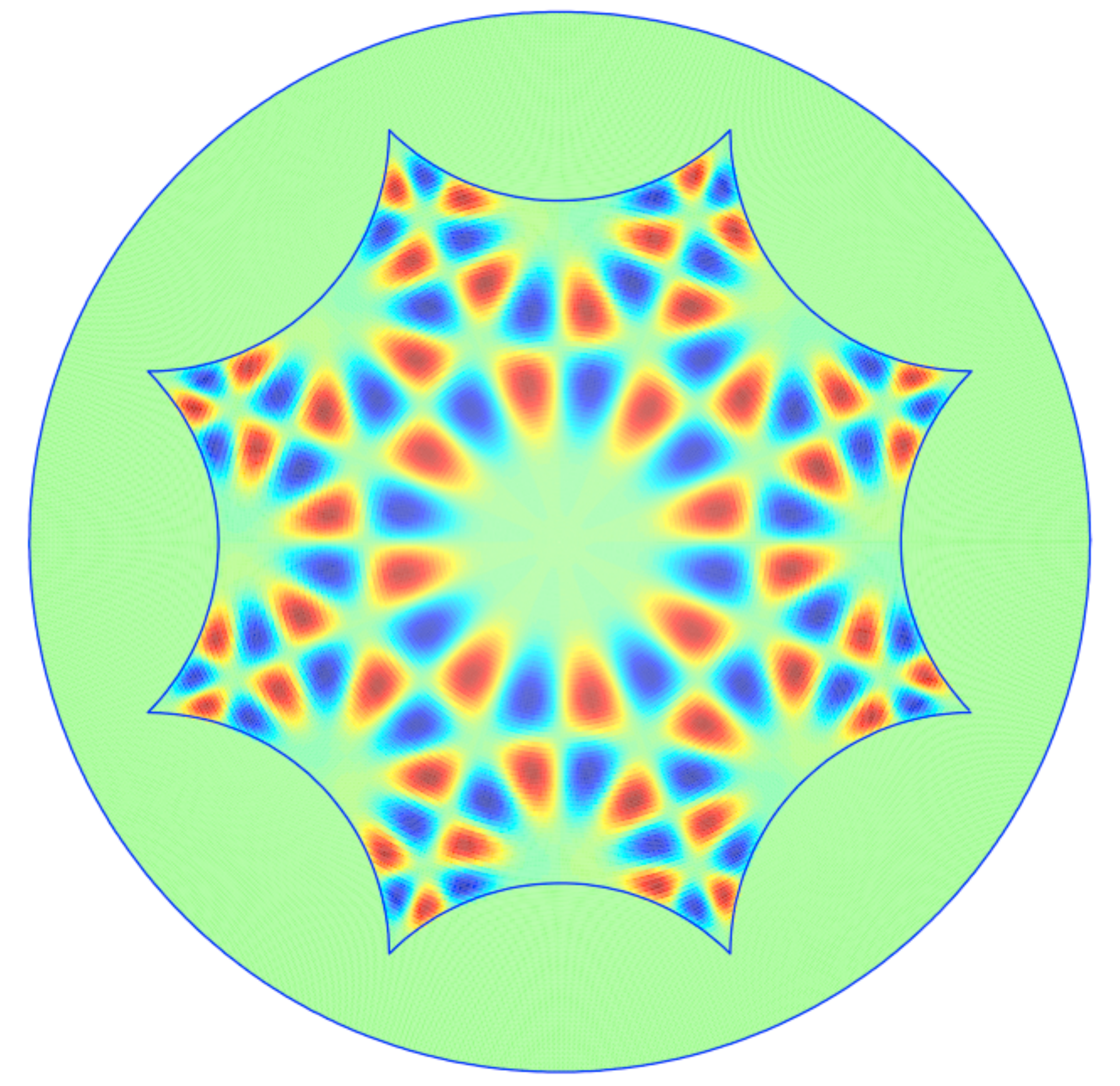}}
\caption{The four H-planforms with their corresponding eigenvalue associated with the four irreductible representations of dimension 1, see text.}
\label{fig:dim1}
\end{figure}

\paragraph{Non desymmetrized problem :} As discussed previously, we also present some H-planforms of higher dimension. We mesh the full octagon with 3641 nodes in such a way that the resulting mesh enjoys a $D_8$-symmetry, see figure \ref{fig:mesh}.

\begin{figure}[htp]
\centering
\includegraphics[width=0.7\textwidth]{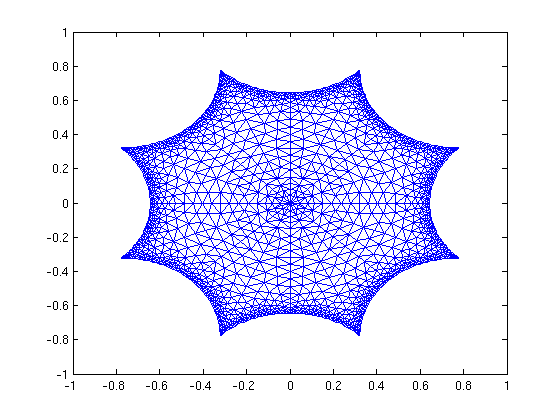}
\caption{A mesh of the octagon with 3641 nodes which is used for the method of finite elements. It leads to matrices of dimension $3641\times3641$.}
\label{fig:mesh}
\end{figure}

We implement, in the finite element method of order 1, the periodic boundary conditions of the eigenproblem and obtain the first 100 eigenvalues of the octagon. Our results are in agreement with those of Aurich and Steiner reported in \cite{aurich-steiner:89}. Instead of giving a table of all the eigenvalues, we prefer to plot the staircase function $N(\lambda)=\sharp\{ \lambda_n | \lambda_n\leq \lambda\}$ for comparison with Weyl's law. Weyl's law is, in its simplest version, a statement on the asymptotic growth of the eigenvalues of the Laplace-Beltrami operator on bounded domains. If $\Omega$ is a given bounded domain of $\R^2$, then the staircase function has the following asymptotic behaviour: 
$N(\lambda)=\frac{| \Omega|}{4\pi}\lambda + o(\lambda)$ as $\lambda \rightarrow \infty$. We recall that in the case of the hyperbolic octagon, one has $|\Omega|=4\pi$ and hence $N(\lambda)\sim \lambda$ as $\lambda \rightarrow \infty$. As can be seen in figure \ref{fig:staircase} the asymptotic law describes the staircase well down to the smallest eigenvalues, which confirms the validity of our numerical results.

\begin{figure}[htp]
\centering
\includegraphics[width=0.7\textwidth]{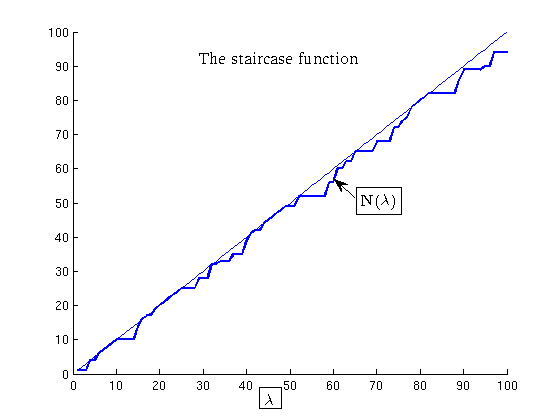}
\caption{The staircase function $N(\lambda)$, in dark, is shown in comparison with Weyl's law $N(\lambda)\sim \lambda$ as $\lambda \rightarrow \infty$.}
\label{fig:staircase}
\end{figure}

We show in figure \ref{fig:D8} two H-planforms, with $\widetilde{D}_8$ and $\widetilde{D}_{8\kappa}$ isotropy respectively. These two H-planforms belong to irreductible representations of dimension 2: $\chi_5$ for \ref{fig:eig73} and $\chi_6$ for \ref{fig:eig10} and \ref{fig:eig2}.

\begin{figure}[htp]
\centering
\subfigure[$\chi_5$ : $\widetilde{D}_8$, the corrseponding eigenvalue is $\lambda=73.7323$.]{
\label{fig:eig73}
\includegraphics[width=0.4\textwidth]{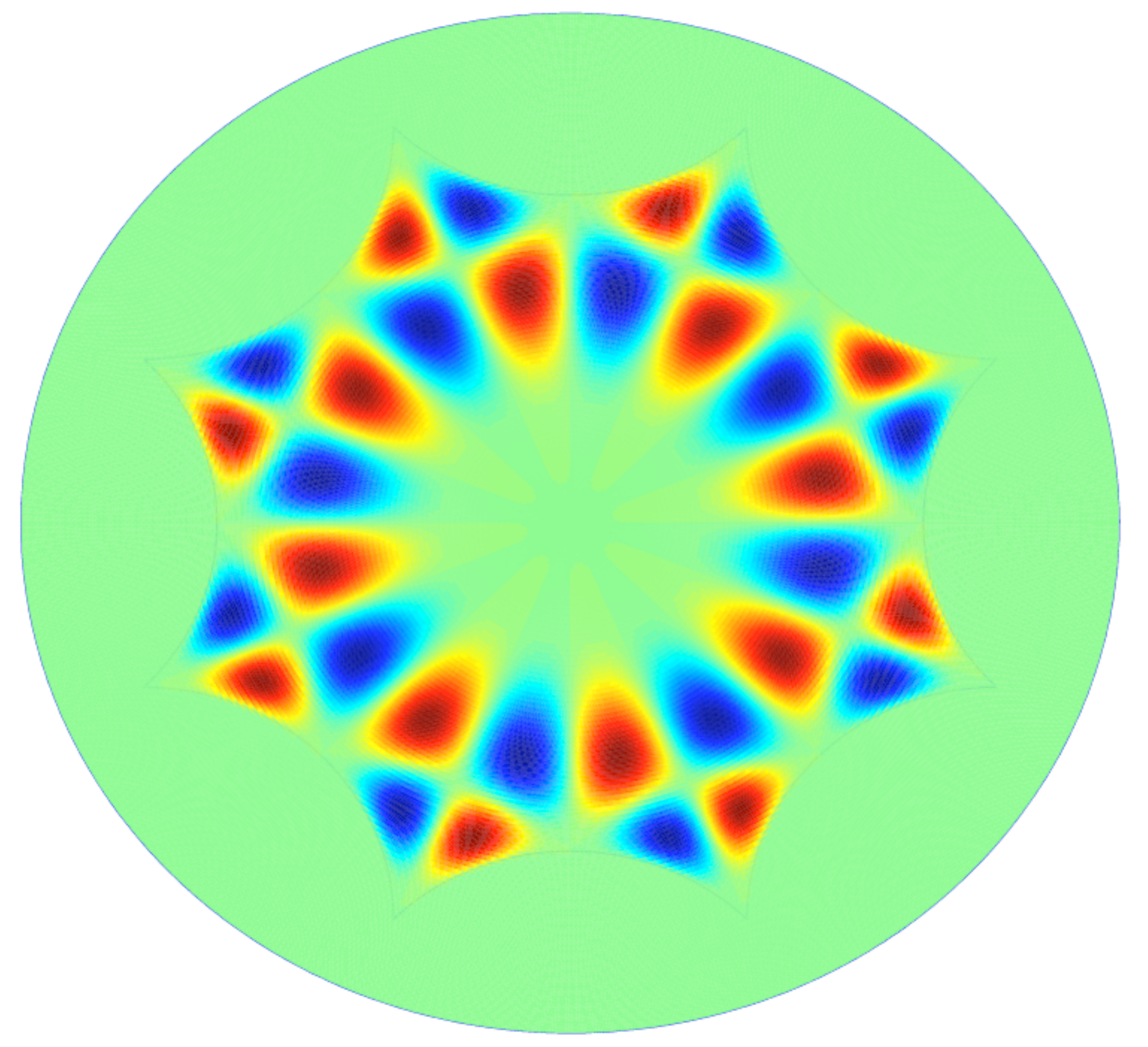}}
\hspace{.3in}
\subfigure[$\chi_6$ : $\widetilde{D}_{8\kappa}$, the corrseponding eigenvalue is $\lambda=8.2501$.]{
\label{fig:eig10}
\includegraphics[width=0.4\textwidth]{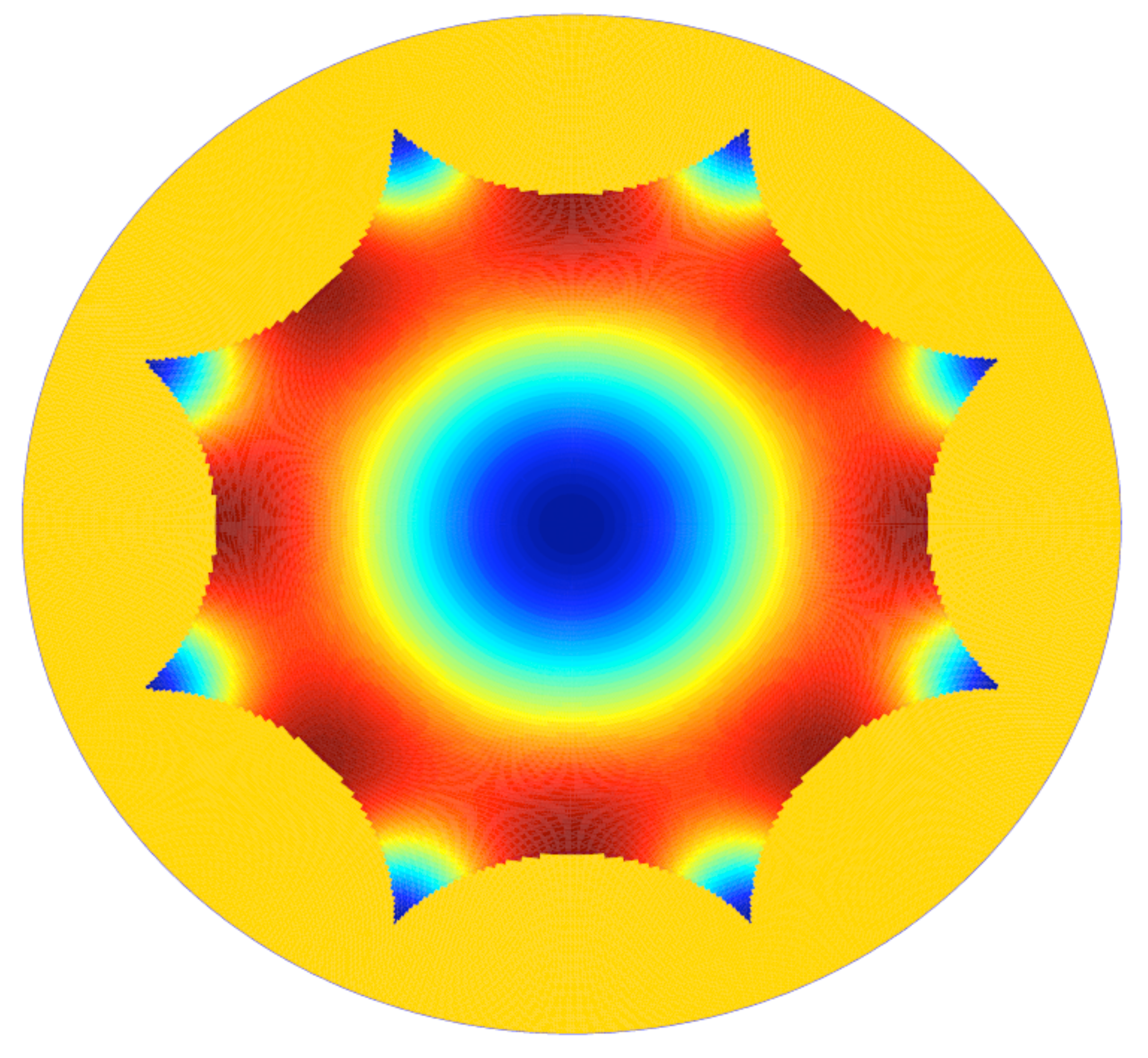}}\\
\caption{Two H-planforms with their corresponding eigenvalue associated to the two irreductible representations of dimension 2, see text.}
\label{fig:D8}
\end{figure}

We finally present in figure \ref{fig:D4} three H-planforms, with $C_{8\kappa}$, $\widetilde{D}_{4\kappa}$ and $\widetilde{D}_{4\kappa'}$ isotropy. These three H-Planforms belong to irreductible representations of dimension 3: $\chi_8$ for \ref{fig:eig2}, $\chi_9$ for \ref{fig:eig28} and $\chi_{10}$ for \ref{fig:eig15}.\\

\begin{figure}[htp]
\centering
\subfigure[$\chi_8$ : $C_{8\kappa}$, the corrseponding eigenvalue is $\lambda=3.8432$.]{
\label{fig:eig2}
\includegraphics[width=0.4\textwidth]{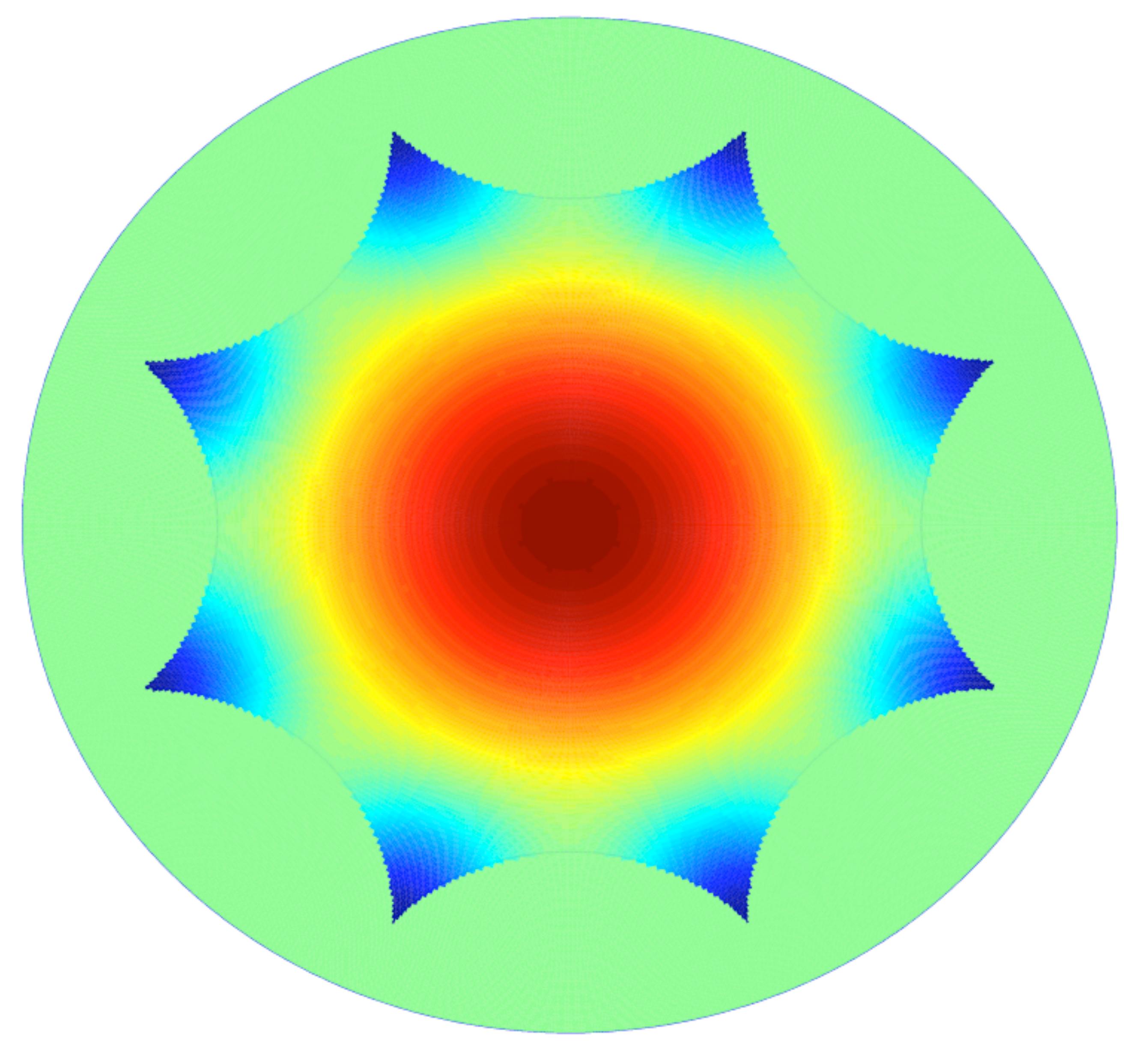}}
\hspace{.3in}
\subfigure[$\chi_9$ : $\widetilde{D}_{4\kappa}$, the corresponding eigenvalue is $\lambda=28.0888$.]{
\label{fig:eig28}
\includegraphics[width=0.37\textwidth]{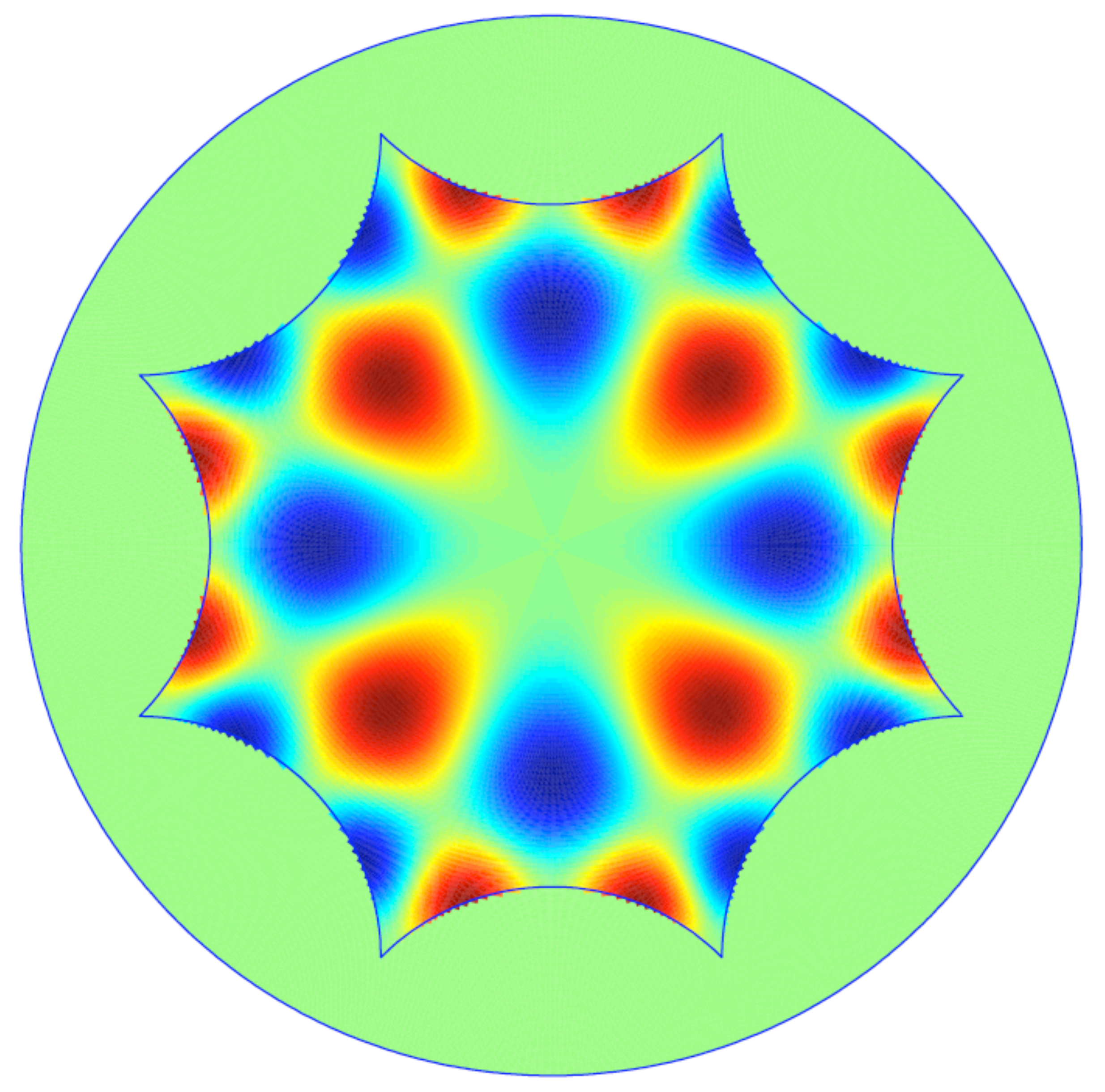}}\\
\subfigure[$\chi_{10}$ : $\widetilde{D}_{4\kappa'}$, the corrseponding eigenvalue is $\lambda=15.0518$.]{
\label{fig:eig15}
\includegraphics[width=0.4\textwidth]{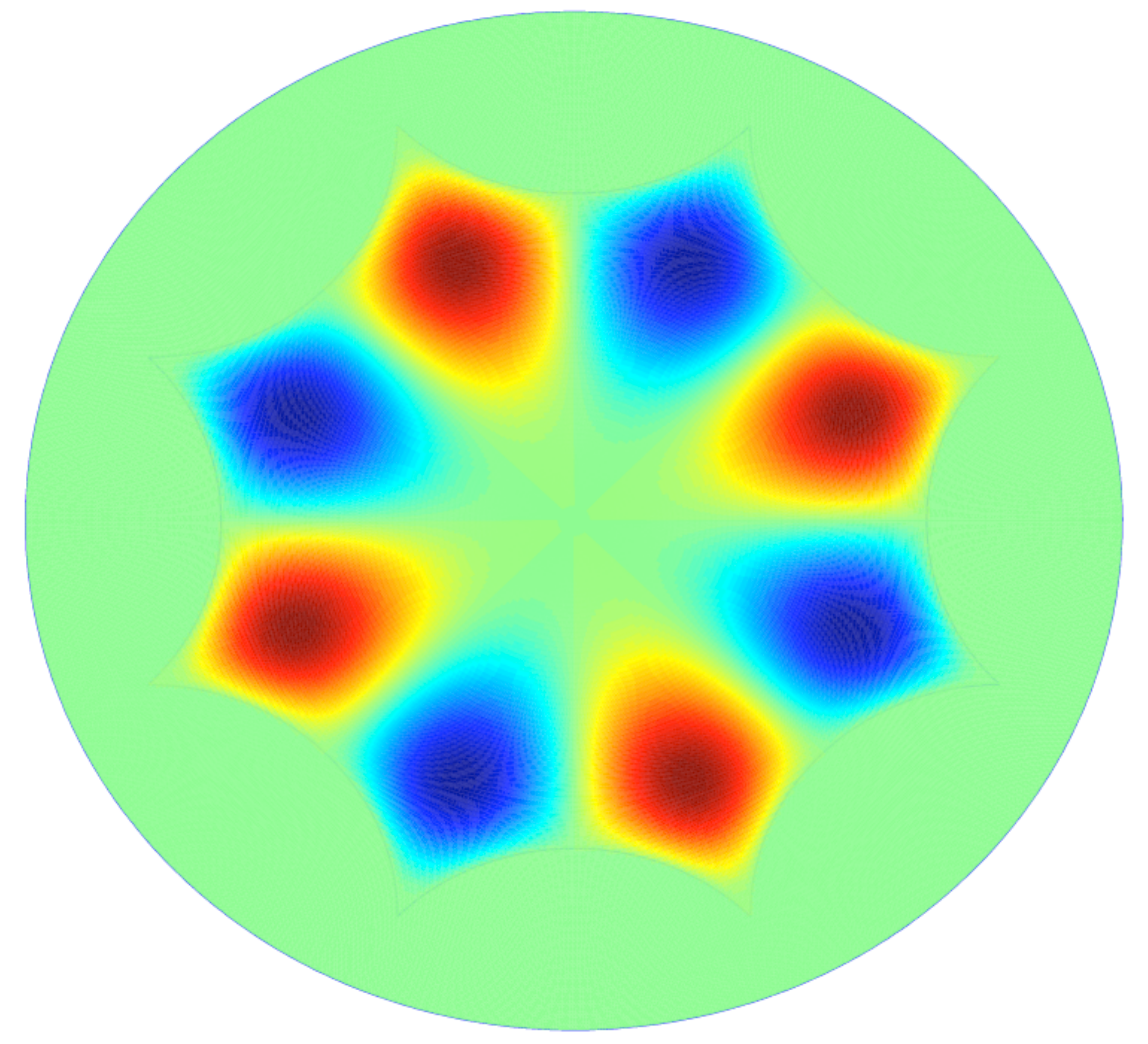}}
\caption{Three H-planforms with their corresponding eigenvalue associated to three irreductible representations of dimension 3, see text.}
\label{fig:D4}
\end{figure}

In figures \ref{fig:dim1},\ref{fig:D8} and \ref{fig:D4}, we have plotted, for convenience, the corresponding H-planforms in the octagon only. Nevertheless, H-planforms are periodic in the Poincar\'e disc, as stated before, and in figure \ref{fig:eigall33}, we plot the H-planform with $G_{0\kappa'}$ isotropy type of figure \ref{fig:chi3}. We recall that the octagonal lattice group $\Gamma$ is generated by the four boosts $g_j$ of subsection \ref{subsection:octlattice}. Then, once the H-planform is calculated, we report it periodically in the whole Poincar\'e disc by the actions of there four boosts and obtain figure \ref{fig:eigall33}. Note that it is arduous to tesselate the entire disc and this is why there remains some untesselated areas in the figure.\\

\begin{figure}[htp]
\centering
\includegraphics[width=0.4\textwidth]{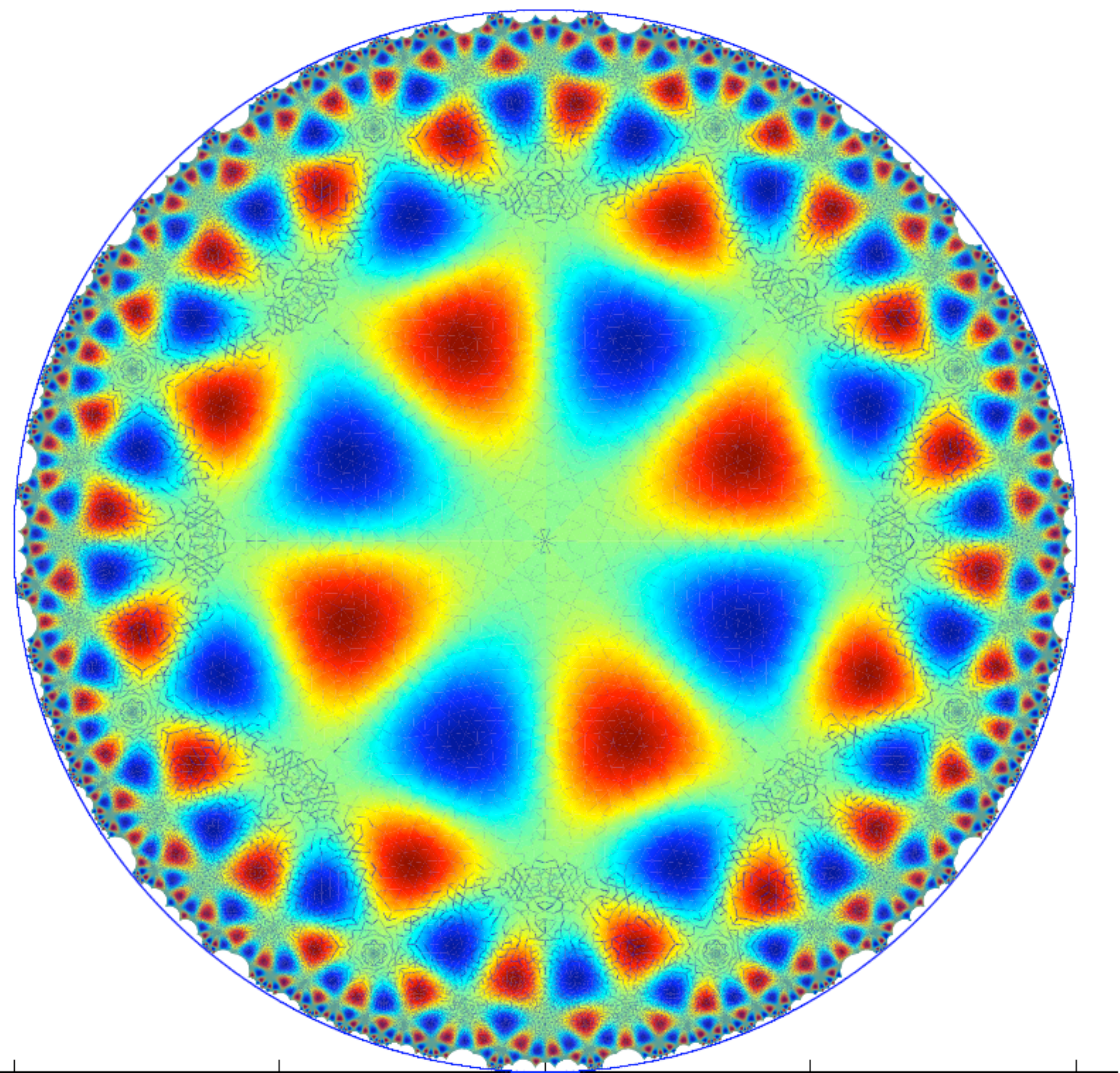}
\caption{Extension of the $G_{0\kappa'}$ H-planform on the Poincar\'e disk.}
\label{fig:eigall33}
\end{figure}

{\em Remark :} The finite element method relies on a variational principle and thus yields upper bounds for the eigenvalues. Note that it is not well-suited to compute high eigenvalues. Thus, if one wants to compute, for example, the first 125 eigenvalues, then the finite element method provides a very good approximation to the true eigenvalues. If however one wants to reach the 2000th eigenvalue by the finite element method we need to use matrices of size at least $10^6\times 10^6$ which is untractable for desktop or laptop computers and requires going to more powerful architectures.
This is why the direct boundary element method should be preferred for such computations (see  \cite{aurich-steiner:93}).


\section{Conclusion}

In this paper, we have analysed the bifurcation of periodic patterns for neural field equations describing the state of a system defined on the space of structure tensors, when these equations are further invariant under the isometries in this space. We have made use of the concept of periodic lattice in $\D$ (Poincar\'e disc) to further reduce the problem to one on a compact Riemann surface $\D/\Gamma$, where $\Gamma$ is a cocompact, torsion-free Fuchsian group. Successfully, we have applied the machinery of equivariant bifurcation theory in the case of an octagonal periodic pattern, where we have been able to classify all possible H-planforms satisfying the hypotheses of the Equivariant Branching Lemma. In the last section, we have described a method to compute these patterns and illustrated it with a selection of images of octagonal H-planforms.

There are several questions which are raised by this analysis.
\begin{itemize}
\item[(i)] The first one is that of the interpretation of these H-planforms for the modelling of the visual cortex.  At this moment we believe that they could be involved in the process of defining texture tuning curves in a way that would resemble the definition of orientation tuning curves in  the related ring model of orientation selectivity \cite{hansel-sompolinsky:97,shriki-hansel-etal:03,ermentrout:98,dayan-abbott:01,bressloff-bressloff-etal:00,bressloff-cowan-etal:01, veltz-faugeras:10}. Another fascinating  possibility is that they could be related to neural illusions caused by the existence of several stable stationary solutions to \eqref{eq:neuralmass} when, e.g., the slope of the sigmoid at the origin becomes larger than that at which several branches of solutions bifurcate from the trivial one. These neural illusions would be functions of membrane potential values that would not correspond to the actual thalamic input and could be expressed as combinations of these H-planforms. This is of course still very speculative but very much worth investigating.
\item[(ii)] The second one is about the observability of such patterns in a natural system or under direct simulation of the evolution equations. Indeed, not only there is a high degeneracy of the bifurcation problem if one removes the assumption that all perturbations respect the periodicity of the pattern, which is also the case for patterns in Euclidean space, but in addition the fact that such patterns would be neutral modes for the bifurcation problem posed in full generality in $\D$ is non generic. We may however imagine mechanisms such as "spatial frequency locking" by which periodic patterns could become stable, hence observable.
\item[(iii)] The third one is that of a more effective computation of H-planforms for a given isotropy type. As we have seen in the last section, a desymmetrization of the domain allowed us to calculate all the H-planforms with isotropy types associated to irreductible representations of dimension 1. For irreductible representations of dimension $\geq 2$ the computation becomes more intricate as the desymmetrization method is no longer straightforward and remains misunderstood. Naturally, one solution would be to elaborate a general algorithm which, for a given isotropy type, computes systematically the associated H-planform. We think that such an algorithm would be of interests for quantum physisists and cosmologists.
\end{itemize}

These questions will be the subject of further studies. \\

\noindent
{\bf Acknowledgments}\\ 
This work was partially funded by the ERC advanced grant NerVi number 227747.


\begin{thebibliography}{10}

\bibitem{allaire:05}
G.~Allaire.
\newblock {\em Analyse num{\'e}rique et optimisation}.
\newblock {\'E}d. de l'Ecole Polytechnique, 2005.

\bibitem{aurich-steiner:89}
R.~Aurich and F.~Steiner.
\newblock Periodic-orbit sum rules for the hadamard-gutzwiller model.
\newblock {\em Physica D}, 39:169--193, 1989.

\bibitem{aurich-steiner:93}
R.~Aurich and F.~Steiner.
\newblock Statistical properties of highly excited quantum eigenstates of a
 strongly chaotic system.
\newblock {\em Physica D}, 64:185--214, 1993.

\bibitem{balazs-voros:86}
N.L. Balazs and A.~Voros.
\newblock Chaos on the pseudosphere.
\newblock {\em Physics Reports}, 143(3):109--240, 1986.

\bibitem{bressloff-bressloff-etal:00}
P.C. Bressloff, N.W. Bressloff, and J.D. Cowan.
\newblock Dynamical mechanism for sharp orientation tuning in an
 integrate-and-fire model of a cortical hypercolumn.
\newblock {\em Neural computation}, 12(11):2473--2511, 2000.

\bibitem{bressloff-cowan-etal:01}
P.C. Bressloff, J.D. Cowan, M.~Golubitsky, P.J. Thomas, and M.C. Wiener.
\newblock Geometric visual hallucinations, euclidean symmetry and the
 functional architecture of striate cortex.
\newblock {\em Phil. Trans. R. Soc. Lond. B}, 306(1407):299--330, mar 2001.

\bibitem{bressloff-cowan-etal:02}
P.C. Bressloff, J.D. Cowan, M.~Golubitsky, P.J. Thomas, and M.C. Wiener.
\newblock {What Geometric Visual Hallucinations Tell Us about the Visual
 Cortex}.
\newblock {\em Neural Computation}, 14(3):473--491, 2002.

\bibitem{broughton:91}
S.A. Broughton.
\newblock Classifying finite group actions on surfaces of low genus.
\newblock {\em Journal of Pure and Applied Algebra}, 69(3):233--270, 1991.

\bibitem{broughton-dirks-etal:01}
S.A. Broughton, R.M. Dirks, M.T. Sloughter, and C.R. Vinroot.
\newblock Triangular surface tiling groups for low genus.
\newblock Technical report, MSTR, 2001.

\bibitem{buser:92}
P.~Buser.
\newblock {\em Geometry and spectra of compact Riemann surfaces}, volume 106.
\newblock Springer, 1992.

\bibitem{chossat-faugeras:09}
P.~Chossat and O.~Faugeras.
\newblock Hyperbolic planforms in relation to visual edges and textures
 perception.
\newblock {\em Plos Comput Biol}, 5(12):e1000625, dec 2009.

\bibitem{chossat-lauterbach:00}
P.~Chossat and R.~Lauterbach.
\newblock {\em {Methods in Equivariant Bifurcations and Dynamical Systems}}.
\newblock World Scientific Publishing Company, 2000.

\bibitem{ciarlet-lions:91}
P.G. Ciarlet and J.L. Lions, editors.
\newblock {\em Handbook of Numerical Analysis. Volume II. Finite Element
 Methods (part1)}.
\newblock North-Holland, 1991.

\bibitem{cornish-spergel:99}
N.J. Cornish and D.N Spergel.
\newblock On the eigenmodes of compact hyperbolic 3-manifolds.
\newblock Technical report, arXiv, 1999.

\bibitem{cornish-turok:98}
N.J Cornish and N.G Turok.
\newblock Ringing the eigenmodes from compact manifolds.
\newblock Technical report, arXiv, 1998.

\bibitem{dayan-abbott:01}
P.~Dayan and L.F. Abbott.
\newblock {\em Theoretical Neuroscience : Computational and Mathematical
 Modeling of Neural Systems}.
\newblock MIT Press, 2001.

\bibitem{dionne-golubitsky:92}
B.~Dionne and M.~Golubitsky.
\newblock Planforms in two and three dimensions.
\newblock {\em ZAMP}, 43:36--62, 1992.

\bibitem{ermentrout:98}
Bard Ermentrout.
\newblock Neural networks as spatio-temporal pattern-forming systems.
\newblock {\em Reports on Progress in Physics}, 61:353--430, 1998.

\bibitem{fassler-stiefel:92}
A.~F{\"a}ssler and E.L. Stiefel.
\newblock {\em Group theoretical methods and their applications}.
\newblock Birkh{\"a}user, 1992.

\bibitem{Gap}
GAP.
\newblock {\em Groups, Algorithms and programming}.
\newblock \href{http://www.gap-system.org/}{URL of GAP}



\bibitem{golubitsky-stewart-etal:88}
M.~Golubitsky, I.~Stewart, and D.G. Schaeffer.
\newblock {\em {Singularities and Groups in Bifurcation Theory}}.
\newblock Springer, 1988.

\bibitem{gorenstein:80}
D.~Gorenstein.
\newblock {\em Finite groups}.
\newblock Chelsea Pub Co, 1980.

\bibitem{hansel-sompolinsky:97}
D.~Hansel and H.~Sompolinsky.
\newblock Modeling feature selectivity in local cortical circuits.
\newblock {\em Methods of neuronal modeling}, pages 499--567, 1997.

\bibitem{helgason:00}
S.~Helgason.
\newblock {\em Groups and geometric analysis}, volume~83 of {\em Mathematical
 Surveys and Monographs}.
\newblock American Mathematical Society, 2000.

\bibitem{hoyle:06}
R.B. Hoyle.
\newblock {\em Pattern formation: an introduction to methods}.
\newblock Cambridge Univ Pr, 2006.

\bibitem{inoue:99}
K.T Inoue.
\newblock Computation of eigenmodes on a compact hyperbolic 3-space.
\newblock Technical report, arXiv, 1999.

\bibitem{iwaniec:02}
H.~Iwaniec.
\newblock {\em Spectral methods of automorphic forms}, volume~53 of {\em {AMS}
 Graduate Series in Mathematics}.
\newblock {AMS} Bookstore, 2002.

\bibitem{katok:92}
S.~Katok.
\newblock {\em Fuchsian Groups}.
\newblock Chicago Lectures in Mathematics. The University of Chicago Press,
 1992.

\bibitem{lang:93}
S.~Lang.
\newblock {\em Algebra}.
\newblock Addison-Wesley, third edition edition, 1993.

\bibitem{lehoucq-weeks-etal:02}
R.~Lehoucq, J.~Weeks, J-P. Uzan, E.~Gausmann, and J-P. Luminet.
\newblock Eigenmodes of 3-dimensional spherical spaces and their application to
 cosmology.
\newblock Technical report, arXiv, 2002.

\bibitem{miller:72}
W.~Miller.
\newblock {\em Symmetry groups and their applications}.
\newblock Academic Press, 1972.

\bibitem{moakher:05}
M.~Moakher.
\newblock A differential geometric approach to the geometric mean of symmetric
 positive-definite matrices.
\newblock {\em SIAM J. Matrix Anal. Appl.}, 26(3):735--747, April 2005.

\bibitem{pollicott:89}
M.~Pollicott.
\newblock Distributions at infinity for riemann surfaces.
\newblock In Stefan~Banach Center, editor, {\em Dynamical Systems and Ergodic
 Theory}, volume~23, pages 91--100, 1989.

\bibitem{sausset-tarjus:07}
F.~Sausset and G.~Tarjus.
\newblock Periodic boundary conditions on the pseudosphere.
\newblock {\em Journal of Physics A: Mathematical and Theoretical},
 40:12873--12899, 2007.

\bibitem{schmit:91}
C.~Schmit.
\newblock Quantum and classical properties of some billiards on the hyperbolic
 plane.
\newblock {\em Chaos and Quantum Physics}, pages 335--369, 1991.

\bibitem{series:87}
C.~Series.
\newblock Some geometrical models of chaotic dynamics.
\newblock {\em Proceedings of the Royal Society of London. Series A,
 Mathematical and Physical Sciences}, 413(1844):171--182, 1987.

\bibitem{serre:78}
J.P. Serre.
\newblock {\em Repr{\'e}sentations lin{\'e}aires des groupes finis}.
\newblock Hermann, 1978.

\bibitem{shriki-hansel-etal:03}
O.~Shriki, D.~Hansel, and H.~Sompolinsky.
\newblock Rate models for conductance-based cortical neuronal networks.
\newblock {\em Neural Computation}, 15(8):1809--1841, 2003.

\bibitem{veltz-faugeras:10}
R.~Veltz and O.~Faugeras.
\newblock Local/global analysis of the stationary solutions of some neural field equations.
\newblock {\em Accepted for publication in SIADS}, 2010.

\bibitem{weil:57}
A.~Weil.
\newblock Groupes des formes quadratiques ind{\'e}finies et des formes
 bilin{\'e}aires altern{\'e}es.
\newblock {\em S{\'e}minaire Henri Cartan}, 10:1--14, 1957.

\end{thebibliography}
\end{document}